\documentclass[onefignum,onetabnum]{siamonline220329}

\usepackage{lipsum}
\usepackage{amsfonts}

\ifpdf
  \DeclareGraphicsExtensions{.eps,.pdf,.png,.jpg}
\else
  \DeclareGraphicsExtensions{.eps}
\fi

\usepackage{enumitem}
\setlist[enumerate]{leftmargin=.5in}
\setlist[itemize]{leftmargin=.5in}

\newsiamremark{remark}{Remark}
\newsiamremark{hypothesis}{Hypothesis}
\crefname{hypothesis}{Hypothesis}{Hypotheses}
\newsiamthm{claim}{Claim}

\headers{Randomized TR Decomposition}{Yajie Yu, and Hanyu Li}

\title{Practical Sketching-Based Randomized Tensor Ring Decomposition\thanks{
This work was funded by the National Natural Science Foundation of China (No. 11671060) and the Natural Science Foundation of Chongqing, China (No. cstc2019jcyj-msxmX0267).}}

\author{Yajie Yu\thanks{College of Mathematics and Statistics, Chongqing University, Chongqing 401331, P.R. China
  (\email{zqyu@cqu.edu.cn;\ lihy.hy@gmail.com or hyli@cqu.edu.cn}).}
\and Hanyu Li
\footnotemark[2]}

\usepackage{amsopn}

\usepackage{amssymb,amsfonts,amsmath}
\usepackage[mathscr]{eucal}
\usepackage{multirow}
\usepackage{makecell}
\usepackage{amsbsy}
\usepackage{stmaryrd}
\usepackage{bm}
\usepackage{braket}  
\usepackage{xparse}
\usepackage{xspace}
\usepackage{xcolor}
\usepackage{hyperref}
\usepackage{cleveref}
\usepackage{extarrows}
\usepackage{booktabs}
\usepackage{microtype}
\usepackage{algorithm}  
\usepackage{algorithmicx}  
\usepackage{algpseudocode}
\usepackage {graphicx,epstopdf} 
\usepackage {array}
\usepackage[caption=false]{subfig}
\usepackage{tablefootnote} 
\DeclareMathOperator{\FFT}{FFT}
\DeclareMathOperator{\TR}{TR}
\DeclareMathOperator{\trace}{trace}

\newcommand{\bfA}{{\bf A}}
\newcommand{\bfB}{{\bf B}}

\newcommand{\bfD}{{\bf D}}

\newcommand{\bfF}{{\bf F}}
\newcommand{\bfG}{{\bf G}}

\newcommand{\bfM}{{\bf M}}

\newcommand{\bfQ}{{\bf Q}}

\newcommand{\bfS}{{\bf S}}
\newcommand{\bfT}{{\bf T}}
\newcommand{\bfU}{{\bf U}}

\newcommand{\bfX}{{\bf X}}

\newcommand{\bfZ}{{\bf Z}}

\newcommand{\bfa}{{\bf a}}
\newcommand{\bfb}{{\bf b}}

\newcommand{\mcP}{{\mathcal{P}}}

\newcommand{\mcS}{{\mathcal{S}}}

\newcommand{\bbC}{{\mathbb{C}}}

\newcommand{\bbR}{{\mathbb{R}}}

\newcommand{\mfI}{{\mathfrak{I}}}

\newcommand{\mfR}{{\mathfrak{R}}}

\newcommand{\tensor}[1]{\boldsymbol{\mathcal{#1}}}

\newcommand{\mat}[1]{\mathbf{#1}}
\newcommand{\vect}[1]{\bm{#1}}
\newcommand{\bigO}[1]{\mathcal{O}\left( #1 \right)}

\newcommand{\LineRef}[1]{\hyperref[#1]{Line~\ref{#1}}}
\ifpdf
\hypersetup{
  pdftitle={Practical Sketching-Based Randomized Tensor Ring Decomposition},
  pdfauthor={Yajie Yu, and Hanyu Li}
}
\fi

\begin{document}
\begin{sloppypar}
\maketitle

\begin{abstract}
Based on sketching techniques, we propose two randomized algorithms for tensor ring (TR) decomposition. Specifically, by defining new tensor products and investigating their properties, we apply the Kronecker sub-sampled randomized Fourier transform and TensorSketch to the alternating least squares problems derived from the minimization problem of TR decomposition to devise the randomized algorithms. From the former, we find an algorithmic framework based on random projection for randomized TR decomposition. Theoretical results on sketch size and complexity analyses for the two algorithms are provided. We compare our proposals with the state-of-the-art method using both synthetic and real data. Numerical results show that they have quite decent performance in accuracy and computing time. 
\end{abstract}



\section{Introduction}
\label{sec:introduction}
Tensor decompositions represent a higher-order tensor by multilinear operations over the latent factors, which have found many applications in machine learning, signal processing, chemometrics and so on; see the detailed review in \cite{kolda2009TensorDecompositions, sidiropoulos2017tensor,cichocki2015TensorDecompositions}. Two popular decompositions are the Canonical Polyadic (CP) decomposition and the Tucker decomposition. The former has the parameters scaling linearly with the tensor order and the latter is easier to work numerically. So these two decompositions have wide applications \cite{kolda2009TensorDecompositions,sidiropoulos2017tensor,cichocki2015TensorDecompositions}. However, they also suffer from some limitations. For example, finding the CP decomposition of a tensor is an NP-hard problem; the number of parameters of Tucker decomposition still scales exponentially with the tensor order. Tensor train (TT) decomposition, which is also known as the matrix product state (MPS) with open boundary conditions in quantum physics \cite{affleck2004ValenceBond,perez-garcia2007MatrixProduct}, came into being to address these challenges \cite{oseledets2011TensorTrainDecomposition}. However, it also has some limitations such as the constraint on TT-ranks, the fixed pattern of TT-ranks, and the strict order of TT-cores \cite{zhao2016TensorRing}. For these reasons, Zhao et al. \cite{zhao2016TensorRing} introduced the tensor ring (TR) decomposition, i.e., the MPS with periodic boundary conditions, which has the circular dimensional permutation invariance and hence can overcome some shortcomings of TT decomposition. 

Specifically, the TR decomposition in the element-wise form of the tensor $\tensor{X} \in \bbR^{I_1 \times I_2 \times \cdots \times I_N}$ can be represented as follows:
\begin{align*}
	\tensor{X}(i_1, \cdots, i_N)&={\trace} \left(\bfG_1(i_1) \bfG_2(i_2) \cdots \bfG_N(i_N)\right)
	={\trace} \left(\prod_{n=1}^N \bfG_n(i_n)\right),
\end{align*}
where $\bfG_n(i_n) = \tensor{G}_n(:,i_n,:) \in \bbR^{R_n \times R_{n+1}}$ is the $i_n$-th \emph{lateral slice} of the \emph{core tensor (TR-core)} $\tensor{G}_n \in \bbR^{R_n \times I_n \times R_{n+1}}$. Note that a \emph{slice} is an 2-order section, i.e., a matrix, of a tensor obtained by fixing all the tensor indices but two, and $ R_{N+1}=R_{1} $. The sizes of TR-cores, i.e., $R_k$ with $k=1,\cdots,N$, are called \emph{TR-ranks}. Additionally, we use the notation $\TR \left( \{\tensor{G}_n\}_{n=1}^N \right)$ to denote the TR decomposition of a tensor. 
The problem of fitting $\TR \left( \{\tensor{G}_n\}_{n=1}^N \right)$ to a tensor $\tensor{X}$ can be written as the following minimization problem:
\begin{equation}
	\label{eq:trmin}
	\mathop{\arg\min}_{\tensor{G}_1, \cdots, \tensor{G}_N} \| \TR \left( \{\tensor{G}_n\}_{n=1}^N \right) - \tensor{X} \|_F,
\end{equation} 
where $\| \cdot \|_F$ denotes the Frobenius norm of a matrix or tensor. One of the common approaches to this problem is SVD based, and another one uses the alternating least squares (ALS) \cite{zhao2016TensorRing}. We use the abbreviations TR-SVD and TR-ALS for these two methods in the following text. 

Due to the expensive computational costs of the above two methods, it is crucial to design efficient randomized algorithms for computing the TR decompositions of large-scale tensors.
As far as we know, the first randomized algorithm for TR decomposition was proposed by Yuan et al. \cite{yuan2019RandomizedTensor} who first applied the randomized Tucker decomposition to the original tensor and then the TR-SVD or TR-ALS to the core tensor of the Tucker decomposition. Finally, the original TR-cores were recovered by combining the factor matrices of the randomized Tucker decomposition and the TR-cores of the core tensor.
Later, Ahmadi-Asl et al. \cite{ahmadi-asl2020RandomizedAlgorithms} summarized some algorithms of TR decomposition and presented several randomized algorithms with randomized SVD. Recently, building on TR-ALS, Malik and Becker \cite{malik2020SamplingBased} provided a sampling method based on leverage scores, which can avoid forming the large coefficient matrices in TR-ALS via the multi-index and can outperform the randomized method given in \cite{yuan2019RandomizedTensor}. Moreover, Malik \cite{malik2021MoreEfficient} also provided a new approach to approximate the leverage scores and devised the corresponding sampling algorithm for TR decomposition.

As we know, there are two main types of methods for designing randomized algorithms \cite{drineas2017LecturesRandomized,clarkson2017LowRankApproximation,martinsson2020RandomizedNumerical,drineas2011FasterLeast,woodruff2014SketchingTool}. One is \emph{random sampling}, and the other is \emph{random projection}. The latter is often referred to as the sketching-based method. Actually, both of the two types of methods can be regarded as sketching-based methods because the processed matrix in these two methods can be written in a unified form. That is, the matrix is multiplied by a matrix with random entries.  Besides TR decomposition, the above two types of methods have also been applied to other tensor decompositions such as CP, Tucker, and TT decompositions.
For example, Battaglino et al. \cite{battaglino2018PracticalRandomized} first investigated the uniform sampling for CP decomposition comprehensively, and then Larsen and Kolda \cite{larsen2020PracticalLeverageBased} considered its leverage-based sampling. The latter inspires the works in \cite{malik2020SamplingBased,malik2021MoreEfficient} on TR decomposition mentioned above\footnote{In \cite{malik2021MoreEfficient}, the author also provided the sampling algorithm for CP decomposition.}. Again, it should be emphasized that the sampling methods used in these works  can avoid forming the large coefficient matrices and they carry out sample selection in factor matrices or core tensors. For random projection, scholars mainly used Gaussian random matrix, sub-sampled randomized Hadamard transform (SRHT), sub-sampled randomized Fourier transform (SRFT),  CountSketch (CS) and its variant TensorSketch to design the randomized algorithms for CP, Tucker, and TT decompositions; see, e.g., \cite{zhou2014DecompositionBig,battaglino2018PracticalRandomized,che2019RandomizedAlgorithms,minster2020RandomizedAlgorithms,che2020ComputationLow,che2021EfficientRandomized,ma2021FastAccurate}. In particular, in \cite{battaglino2018PracticalRandomized}, the SRFT is applied to each factor matrix of CP decomposition rather than the large coefficient matrix in the ALS (CP-ALS) derived from the minimization problem of CP decomposition; in \cite{ma2021FastAccurate}, the method is mainly built on a formula that combines the TensorSketch and the Kronecker-like form of coefficient matrix. 

In this paper, motivated by \cite{malik2020SamplingBased}  and the warning in \cite[Section 9]{martinsson2020RandomizedNumerical} that leverage score sampling is rarely a competitive method in practice and may perform markedly worse, we consider the sketching-based methods for TR decomposition on the basis of TR-ALS.
Specifically, in our algorithm, we will use Kronecker SRFT (KSRFT), which is similar to what was done in \cite{battaglino2018PracticalRandomized} for CP decomposition but with some quite remarkable improvements.
This is mainly because unlike the coefficient matrices in CP-ALS \cite{battaglino2018PracticalRandomized}, it is difficult to straighten the structure of the coefficient matrices in TR-ALS; see \cref{eq:tr_als} below. To this end, we define a new tensor product and find an elegant property of the product. Moreover, these new findings make us avoid forming the
full coefficient matrices, the large sketching matrices and the matrix multiplications between them, and enable us to detect a framework for designing the randomized algorithm for TR decomposition; see \cref{rem:equivalent} below for details.
Furthermore, as done in \cite{ma2021FastAccurate}, we also consider TensorSketch. 
However,  for our problem, it is more complicated. To this end, we define another new product of tensors and obtain a graceful formula. Therefore, in the ideas on sketching methods, the work is mainly inspired by \cite{battaglino2018PracticalRandomized} and \cite{ma2021FastAccurate} in essence.

The remainder of this paper is organized as follows. \Cref{sec:preliminaries} provides some preliminaries.
In \Cref{sec:main}, we review some existing algorithms for TR decomposition and present our new findings including new tensor products and their properties, structure analysis of the coefficient matrices, and a formula for TensorSketch. The algorithms based on KSRFT and TensorSketch and their theoretical analyses and computational complexities are given in \Cref{sec:srft_tr,sec:ts_tr}, respectively. In \Cref{sec:numerical}, we compare our algorithms with the existing ones using both synthetic and real data. Finally, the concluding remarks of the whole paper are presented. 

\section{Preliminaries}
\label{sec:preliminaries}

We first introduce some necessary definitions.
\begin{definition}[Multi-index \cite{dolgov2014AlternatingMinimal}]
	\label{def:idx}
	For a positive integer $I$, let $[I] \xlongequal{def} \{ 1, \cdots, I \}$. 
	For indices $i_1 \in [I_1], \cdots, i_N \in [I_N]$, a multi-index $i = \overline{i_1 i_2 \cdots i_N}$ refers to an index which takes all possible combinations of values of the indices, $i_1, i_2, \cdots, i_N$, for $i_n = 1,2, \cdots, I_n$ with $n = 1,2, \cdots, N$ in a specific order.  Two common orders are as follows.
	
	\begin{enumerate}
		\item Little-endian convention (reverse lexicographic ordering)
		$$\overline{i_1 i_2 \cdots i_N} = i_1+(i_2-1)I_1 +(i_3-1)I_1 I_2+ \cdots +(i_N-1)I_1 \cdots I_{N-1}.$$  
		\item Big-endian convention (colexicographic ordering)
		$$\overline{i_1 i_2 \cdots i_N} = i_N + (i_{N-1}-1)I_N+(i_{N-2}-1)I_N I_{N-1}+ \cdots +(i_1-1)I_2 \cdots I_N.$$
	\end{enumerate}
	In this paper, unless otherwise stated, we will use the little-endian convention.
\end{definition}

\begin{definition}[Kronecker product]
\label{def:kronecker}
	The \textbf{Kronecker product} of two matrices $\bfA=(a_{ij}) \in \bbR^{I_1 \times J_1}$ and $\bfB \in \bbR^{I_2 \times J_2}$ is a matrix of size $I_1 I_2 \times J_1 J_2$ denoted by $\bfA \otimes \bfB$ and defined as
	\begin{equation*}
	\bfA \otimes \bfB = 
	\begin{bmatrix}
	a_{11}\bfB & \cdots & a_{1J_1}\bfB \\
	\vdots & \ddots & \vdots \\
	a_{I_11}\bfB & \cdots & a_{I_1J_1}\bfB
	\end{bmatrix}.
	\end{equation*}
	
	If using the notation in \Cref{def:idx}, the element-wise form of the Kronecker product is as follows:
	\begin{equation*}
	(\bfA \otimes \bfB) (\overline{i_2 i_1},\overline{j_2 j_1}) = \bfA(i_1,j_1) \bfB(i_2,j_2).
	\end{equation*}
\end{definition}

\begin{definition}[Khatri-Rao product]
\label{def:kr}
	The \textbf{Khatri-Rao product} of two matrices $\bfA \in \bbR^{I_1 \times J}$ and $\bfB \in \bbR^{I_2 \times J}$ is a matrix of size $I_1 I_2 \times J$ denoted by $\bfA \odot \bfB$ and defined as
	\begin{equation*}
	\bfA \odot \bfB = [\bfa_1 \otimes \bfb_1, \bfa_2 \otimes \bfb_2, \cdots, \bfa_J \otimes \bfb_J],
	\end{equation*}
	where $\bfa_n$ and $\bfb_n$ with $n \in [J]$ are the $n$-th columns of the matrices $\bfA$ and $\bfB$, respectively.
	
	If using the notation in \Cref{def:idx}, the element-wise form of the Khatri-Rao product is as follows:
	\begin{equation*}
	(\bfA \odot \bfB) ( \overline{i_2 i_1},j) = \bfA(i_1,j) \bfB(i_2,j).
	\end{equation*}
\end{definition}

\begin{definition}[Mode-$n$ product \cite{cichocki2016TensorNetworks,kolda2009TensorDecompositions}]
	\label{def:moden_product}
	The \textbf{mode-n product} of a tensor $\tensor{X} \in \bbR^{I_1 \times I_2 \cdots \times I_N}$ with a matrix $\bfU \in \bbR^{J \times I_n}$ is a tensor of size $I_1 \times \cdots \times I_{n-1} \times J \times I_{n+1} \times \cdots \times I_N$ denoted by $\tensor{X} \times_n \bfU$ and defined as
	\begin{equation*}
	(\tensor{X} \times_n \bfU)_{i_1 \cdots i_{n-1} j i_{n+1} \cdots i_N} = \sum_{i_n = 1}^{I_n} x_{i_1 i_2 \cdots i_N}u_{j i_n}.
	\end{equation*}
\end{definition}

\begin{definition}[Mode-$n$ unfolding \cite{zhao2016TensorRing,cichocki2016TensorNetworks}]
	\label{def:moden}
	The \textbf{mode-n unfolding} of a tensor $\tensor{X} \in \bbR^{I_1 \times I_2 \cdots \times I_N}$ is the matrix $\bfX_{[n]}$ of size $I_n \times \prod_{j \ne n} I_j$ defined element-wise via
	\begin{equation*}
	\bfX_{[n]}(i_n, \overline{i_{n+1} \cdots i_N i_1 \cdots i_{n-1}})=\tensor{X}(i_1, \cdots, i_N).
	\end{equation*}
\end{definition}

\begin{definition}[Classical mode-$n$ unfolding \cite{zhao2016TensorRing,cichocki2016TensorNetworks,kolda2009TensorDecompositions}]
	The \textbf{classical mode-n unfolding} of a tensor $\tensor{X} \in \bbR^{I_1 \times I_2 \cdots \times I_N}$ is the matrix $\bfX_{(n)}$ of size $I_n \times \prod_{j \ne n} I_j$ defined element-wise via
	\begin{equation*}
	\bfX_{(n)}(i_n, \overline{i_1 \cdots i_{n-1} i_{n+1} \cdots i_N})=\tensor{X}(i_1, \cdots, i_N).
	\end{equation*}
\end{definition}

For the tensor in a series of mode-$n$ product form, we can rewrite its mode-$n$ 
and classical mode-$n$ unfolding matrices using the following formulas.

\begin{proposition}[\cite{cichocki2016TensorNetworks}]
	Let the 
	tensor $\hat{\tensor{X}}$ 
	have the form $\hat{\tensor{X}} = \tensor{X} \times_1 \bfU_1 \times_2 \bfU_2 \cdots \times_N \bfU_N$, where $\tensor{X} \in \bbR^{I_1 \times I_2 \cdots \times I_N}$, and $\bfU_n \in \bbR^{J_n \times I_n}$ for $n = {1, \cdots, N}$. Then
	\begin{align*}
	\hat{\bfX}_{[n]}& = \bfU_n \bfX_{[n]} \left( \bfU_{n-1} \otimes \cdots \otimes \bfU_{1} \otimes \bfU_{N} \otimes \cdots \otimes \bfU_{n+1} \right)^\intercal,\\
	\hat{\bfX}_{(n)} &= \bfU_n \bfX_{(n)} \left( \bfU_N \otimes \cdots \otimes \bfU_{n+1} \otimes \bfU_{n-1} \otimes \cdots \otimes \bfU_1 \right)^\intercal.
	\end{align*}
\end{proposition}

\begin{definition}[Subchain tensor \cite{zhao2016TensorRing}]
	\label{def:subchain}
	Let $\tensor{X} = \TR \left( \{\tensor{G}_n\}_{n=1}^N \right) \in \bbR^{I_1 \times I_2 \cdots \times I_N}$. The \textbf{subchain tensor} $\tensor{G}^{\ne n} \in \bbR^{R_{n+1} \times \prod_{j \ne n} I_j \times R_n}$ is the merging of all TR-cores expect the $n$-th one and can be written slice-wise via
	\begin{equation}
	\label{eq:subchain_idx}
	\bfG^{\ne n}(\overline{i_{n+1} \cdots i_N i_1 \cdots i_{n-1}})=\prod_{j=n+1}^{N} \bfG_j(i_j) \prod_{j=1}^{n-1} \bfG_j(i_j).
	\end{equation}
\end{definition}	

In the following, we introduce KSRFT and TensorSketch. As mentioned in \Cref{sec:introduction}, we will use them to devise the corresponding algorithms. 

\begin{definition}[KSRFT \cite{battaglino2018PracticalRandomized, jin2021FasterJohnson}]\label{def:SRFT}
	The KSRFT is defined as
	\begin{equation*}
	\mathbf{\Phi} = \sqrt{\frac{\prod_{j=1}^N I_j}{m}}
	\bfS  \left(\bigotimes_{j=1}^N (\bfF_j \bfD_j)\right),
	\end{equation*}
	where
	\begin{itemize}
		\item $\bfS \in \bbR^{m \times \prod_{j=1}^N I_j}$ : $m$ rows  of the $\prod_{i=j}^N I_j \times \prod_{j=1}^N I_j$ identity matrix drawn uniformly at random with replacement from the identity matrix;
		\item $\bfF_j \in \bbC^{I_j \times I_j}$ : (unitary) discrete Fourier transform of dimension $I_j$ (also called DFT/FFT matrix);
		\item $\bfD_j \in \bbR^{I_j \times I_j}$ : a diagonal matrix with independent random diagonal entries drawn uniformly from $\{+1,-1\}$ (also called random sign-flip operator).
	\end{itemize}
\end{definition}
\begin{remark}
	Actually, the KSRFT is a special Kronecker fast Johnson-Lindenstrauss transform, i.e., KFJLT, proposed and studied in \cite{battaglino2018PracticalRandomized, jin2021FasterJohnson,malik2020GuaranteesKronecker}.
\end{remark}

\begin{definition}[TensorSketch \cite{paghrasmus2013CompressedMatrix,diao2018sketching}]
\label{def:tensersketch_tr}
	The TensorSketch is defined  as $\mat{T}= \mat{\Omega D}$, where
	\begin{itemize}
		\item $\mat{\Omega} \in \bbR^{m\times \prod_{j=1}^N I_j}$: a matrix with $\mat{\Omega}(j,i) = 1$ if $j=H(i)$ for all $ i\in\left[\prod_{j=1}^N I_j\right]$ and $\mat{\Omega}(j,i) = 0$ otherwise;
		\item $\mat{D} \in \bbR^{\prod_{j=1}^N I_j \times \prod_{i=j}^N I_j}$: a diagonal matrix with $\mat{D}(i,i) = S(i)$.
	\end{itemize}
	In the definitions of $\mat{\Omega}$ and $\mat{D}$,
	\begin{eqnarray*}
		H & :& [I_1] \times [I_2] \times \cdots \times [I_N] \rightarrow [m] : (i_1, \ldots , i_N ) \mapsto
		\left(
		\sum_{n=1}^N
		(H_n(i_n) - 1) \mod m
		\right) + 1,\\
		S & : & [I_1] \times [I_2] \times \cdots \times [I_N] \rightarrow \{-1, 1\} :  (i_1, \ldots , i_N ) \mapsto
		\prod_{n=1}^N S_n(i_n),
	\end{eqnarray*}
	where each $H_n$ for $n\in[N]$ is a 3-wise independent hash map that maps $[I_n]\rightarrow[m]$, and each $S_n$ is a 4-wise independent hash map that maps $[I_n]\rightarrow\{-1,1\}$.
	Recall that a hash map is $k$-wise independent if all the designated $k$ keys are independent random variables.
\end{definition}

\begin{remark}
\label{rem:tensorsketch}
	In \Cref{def:tensersketch_tr}, we use the little-endian convention to compute $$H(i)=H(\overline{i_1i_2 \cdots i_N}) \textrm{ and } S(i)=S(\overline{i_1i_2 \cdots i_N}),$$ 
	while the definition of the TensorSketch in \cite{diao2018sketching,malik2020FastRandomized,malik2018LowRankTucker,ma2021FastAccurate} uses the big-endian convention.
\end{remark}

\section{Existing algorithms and new findings}
\label{sec:main}
We first introduce TR-ALS mentioned in \Cref{sec:introduction} and the main idea of the randomized algorithm proposed in \cite{malik2020SamplingBased}.

According to Theorem 3.5 in \cite{zhao2016TensorRing}, the objective in \cref{eq:trmin} can be rewritten as the following $N$ subproblems
\begin{equation}
\label{eq:tr_als}
\mathop{\arg\min}_{\bfG_{n(2)}} \|\bfG_{[2]}^{\ne n} \bfG_{n(2)}^\intercal-\bfX_{[n]}^\intercal \|_F,\ n=1,\cdots, N.
\end{equation}
The so-called TR-ALS is a method that keeps all cores fixed except the $n$-th one and finds the solution to the least squares problem \cref{eq:tr_als} with respect to it. We summarize TR-ALS in \Cref{alg:tr_als}.

\begin{algorithm}
	\caption{TR-ALS \cite{zhao2016TensorRing}}
	\label{alg:tr_als}
	\begin{algorithmic}[1]\footnotesize
		\Function{$\{\tensor{G}_n\}_{n=1}^N$= TR-ALS}{$\tensor{X}, R_1, \cdots, R_N$} 
		\Comment $\tensor{X}$ is the input tensor
		
		\Comment $R_1, \cdots, R_N$ are the TR-ranks
		\State Initialize cores $\tensor{G}_2, \cdots, \tensor{G}_N$ \label{line:als_init}
		\Repeat
		\For{$n = 1, \cdots, N$}
		\State Compute $\bfG_{[2]}^{\ne n}$ from cores \label{line:als_subchain}
		\State Update $\tensor{G}_n = \mathop{\arg\min}_{\tensor{Z}} \| \bfG_{[2]}^{\ne n} \bfZ_{(2)}^\intercal - \bfX_{[n]}^\intercal \|_F$ \label{line:als_ls}
		\EndFor
		\Until{termination criteria met}
		\State \Return  $\tensor{G}_1, \cdots, \tensor{G}_N$
		\EndFunction
	\end{algorithmic}
\end{algorithm}

Since the size of the matrix $\bfG_{[2]}^{\ne n}$ is $\prod_{j \ne n} I_j \times R_nR_{n+1}$, it is very expensive to solve the least squares problem \cref{eq:tr_als} directly. To tackle this problem, Malik and Becker \cite{malik2020SamplingBased} proposed a random sampling variant of  TR-ALS based on importance sampling. The main advantage of the variant is that it can obtain the sketch with sampling without forming the large matrix $\bfG_{[2]}^{\ne n}$. 
This is mainly due to the authors noting that the $\overline{i_{n+1} \cdots i_N i_1 \cdots i_{n-1}}$-th lateral slice of the subchain tensor $\tensor{G}^{\ne n}$ can be given by a sequence of matrix multiplications, that is, \cref{eq:subchain_idx} in \Cref{def:subchain}.
So they devised an algorithm to find the sampled sketch of the tensor $\tensor{G}^{\ne n}$ and hence of the matrix $\bfG^{\ne n}_{[2]}$ by extracting $m$ lateral slices from each of the cores $\tensor{G}_{j}$. The algorithm of this sampling method is summarized in \Cref{alg:sst}. With this algorithm and the leverage scores or their approximations of the classical mode-$2$ unfolding matrices of TR-cores, the random sampling variant of TR-ALS, i.e.,  TR-ALS-Sampled, can be derived \cite{malik2020SamplingBased}. 
We summarize it in \Cref{alg:tr_als_sampled}.

\begin{algorithm}
	\caption{Sampled Subchain Tensor (SST), summarized from \cite{malik2020SamplingBased}}
	\label{alg:sst}
	\begin{algorithmic}[1]\footnotesize
		\Function{$\tensor{G}^{\ne n}_S$= SST}{$\texttt{idxs}, \tensor{G}_{n+1}, \cdots, \tensor{G}_{N}, \tensor{G}_{1}, \cdots, \tensor{G}_{n-1}$} 
		\Comment $\tensor{G}_n \in \bbR^{R_n \times I_n \times R_{n+1}}, n \in [N]$
		
		\Comment $\texttt{idxs} \in \bbR^{m \times (N-1)}$ is from the set of tuples $\{ i_{n+1}^{(j)}, \cdots, i_{N}^{(j)}, i_{1}^{(j)}, \cdots, i_{n-1}^{(j)} \}$ for $j \in [m]$
		
		\Comment \texttt{idxs} can be retrieved from the sampling matrix $\bfS \in \bbR^{m \times \prod_{k \ne n}I_k}$ or the specific sampling with given probabilities
		\State Let $\tensor{G}^{\ne n}_S$ be a tensor of size $R_{n+1} \times m \times R_{n}$, where every lateral slice is an $R_{n+1} \times R_{n}$ identity matrix
		\For{$k = n+1, \cdots, N, 1, \cdots, n-1$}
		\State $\tensor{G}_{(k)S} \leftarrow \tensor{G}_{k}(:, \texttt{idxs}(:,k), :)$
		\State $\tensor{G}^{\ne n}_S \leftarrow \tensor{G}^{\ne n}_S \boxast_2 \tensor{G}_{(k)S}$
		
		\Comment we use $\boxast_2$ defined in \Cref{def:hada_product} below to simplify the representation
		\EndFor
		\State \Return  $\tensor{G}^{\ne n}_S$
		\EndFunction
	\end{algorithmic}
\end{algorithm}

\begin{algorithm}
	\caption{TR-ALS-Sampled \cite{malik2020SamplingBased}} 
	\label{alg:tr_als_sampled}
	\begin{algorithmic}[1]\footnotesize
	    \Function{$\{\tensor{G}_n\}_{n=1}^N$= TR-ALS-Sampled}{$\tensor{X}, R_1, \cdots, R_N, m$}
	    \Comment $m$ is the sampling size
	    \State Initialize cores $\tensor{G}_2, \cdots, \tensor{G}_N$
	    \State Compute probability distributions $\vect{p}_2, \cdots, \vect{p}_N$ based on leverage scores
	    \Repeat
	    \For{$n = 1, \cdots, N$}
	    \State Draw sampling matrix $\mat{S}$ using the probability distributions of $N-1$ cores
	    \State Retrieve \texttt{idxs} from $\mat{S}$
	    \State $\tensor{G}^{\ne n}_S$ = SST($\texttt{idxs}, \tensor{G}_{n+1}, \cdots, \tensor{G}_{N}, \tensor{G}_{1}, \cdots, \tensor{G}_{n-1}$)
		\State $\bfX_{S[n]}^\intercal \leftarrow \bfS \bfX_{[n]}^\intercal$
		\State Update $\tensor{G}_n = \mathop{\arg\min}_{\tensor{Z}} \| \hat{\bfG}_{S[2]}^{\ne n} \bfZ_{(2)}^\intercal - \hat{\bfX}_{S[n]}^\intercal \|_F$
		\State Update probability distributions $\vect{p}_n$ based on leverage scores
	    \EndFor
	    \Until{termination criteria met}
	    \EndFunction
	\end{algorithmic}
\end{algorithm}

Motivated by  TR-ALS-Sampled and the warning in \cite[Section 9]{martinsson2020RandomizedNumerical}, i.e., sampling can be a bit precarious, and inspired by the idea in \cite{battaglino2018PracticalRandomized}, we would like to combine TR-ALS with the sketching techniques to design the sketching-based algorithms for TR decomposition. To this end, we have to straighten the structure of the matrices $\mcS \bfG^{\ne n}_{[2]}$ and $\bfG^{\ne n}_{[2]}$, where $\mcS$ is some sketching matrix. This is because, like \cite{battaglino2018PracticalRandomized}, we also want to find the sketch $\mcS \bfG^{\ne n}_{[2]}$ without forming $\mcS$ and $\bfG^{\ne n}_{[2]}$.
In the following, we present the specific analysis.

Firstly, note that each row of $\mcS \bfG^{\ne n}_{[2]}$ is the linear combination of the rows of the matrix $\bfG^{\ne n}_{[2]}$, and the rows of the matrix $\bfG^{\ne n}_{[2]}$ are the vectorization of the lateral slices of the subchain tensor $\tensor{G}^{\ne n}$.
Therefore, from the perspective of the subchain tensor, the above process can be carried out from a linear combination of the lateral slices of $\tensor{G}^{\ne n}$. Moreover, if we use the properties of mode-2 product, the above operation can be written as $\tensor{G}^{\ne n} \times_2 \mcS$. We illustrate this transformation and process in \Cref{fig:structure}.

\begin{figure}[htbp]
	\centering  
	\includegraphics[width=0.9\columnwidth]{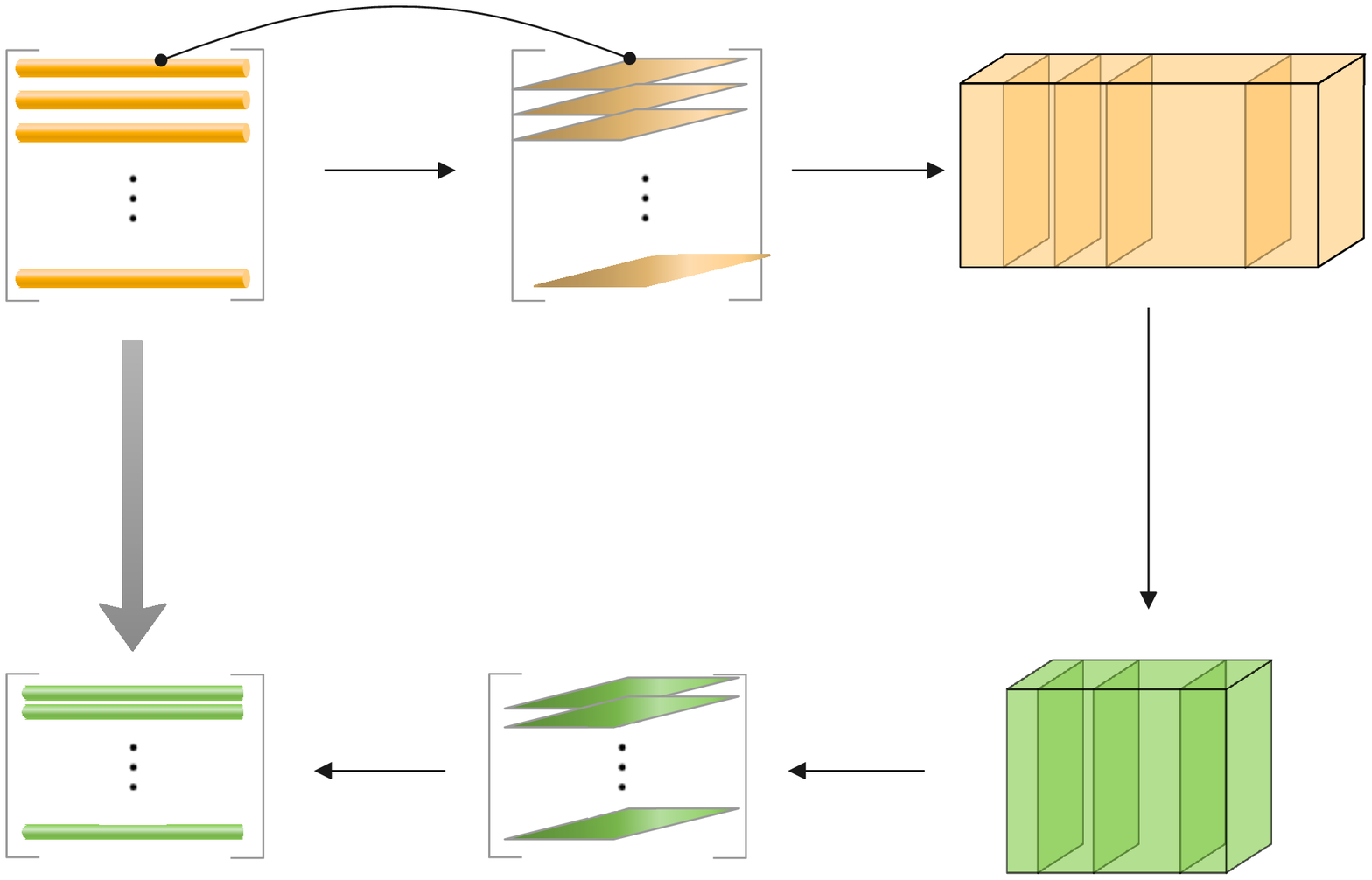}
	\caption{Illustration of the process for obtaining $\mcS \bfG^{\ne n}_{[2]}$ via $\tensor{G}^{\ne n} \times_2 \mcS$.}
	\label{fig:structure}
	\begin{picture}(0,0)
	\put(-210,195){\footnotesize $\bfG_{[2]}^{\ne n} \in \bbR^{\prod_{j \ne n} I_j \times R_nR_{n+1}}$}
	\put(-210,30){\footnotesize $\mcS \bfG_{[2]}^{\ne n} \in \bbR^{m \times R_nR_{n+1}}$}

	\put(-55,195){{\footnotesize $\prod_{j \ne n} I_j $ lateral slides}}
	\put(-55,30){{\footnotesize $m$ lateral slides}}
	
	\put(35,245){{\footnotesize Rotate}}
	
	\put(80,240){\footnotesize  $\tensor{G}^{\ne n} \in \bbR^{R_{n+1} \times \prod_{j \ne n} I_j \times R_n}$}
	\put(95,70){\footnotesize  $\tensor{G}^{\ne n} \times_2 \mcS$}
	\put(95,60){\footnotesize  $\in \bbR^{R_{n+1} \times m \times R_n}$}

	\end{picture}
\end{figure} 

Secondly, from \cref{eq:subchain_idx}, we observe that the subchain tensor $\tensor{G}^{\ne n}$ may be written as a Kronecker-like or Khatri-Rao-like product of TR-cores if we define a tensor product felicitously. The product is given as follows, and its graphical
illustration is provided in \Cref{fig:subchain_product}.
\begin{definition} [Subchain product]
	\label{def:subchain_product}
	Let $\tensor{A} \in \bbR^{I_1 \times J_1 \times K}$ and $\tensor{B} \in \bbR^{K \times J_2 \times I_2}$ be two 3-order tensors, and $\bfA(j_1)$ and $\bfB(j_2)$ be the $j_1$-th and $j_2$-th lateral slices of $\tensor{A}$ and $\tensor{B}$, respectively. The mode-2 \textbf{subchain product} of $\tensor{A}$ and $\tensor{B}$ is a tensor of size $I_1 \times J_1 J_2 \times I_2$ denoted by $\tensor{A} \boxtimes_2 \tensor{B}$ and  defined as 
	\begin{equation*}
	(\tensor{A} \boxtimes_2 \tensor{B})(\overline{j_1 j_2}) = \tensor{A}(j_1)\tensor{B}(j_2).
	\end{equation*}
	That is, with respect to the correspondence on indices, the lateral slices of $\tensor{A} \boxtimes_2 \tensor{B}$ are the classical matrix products of the lateral slices of $\tensor{A}$ and $\tensor{B}$. The mode-1 and mode-3 subchain products can be defined similarly.
\end{definition}

\begin{figure}[htbp] 
 \centering 
 \subfloat[Illustration of the subchain product given in \Cref{def:subchain_product}.]{
 \label{fig:subchain_product}
 \includegraphics[scale=0.5]{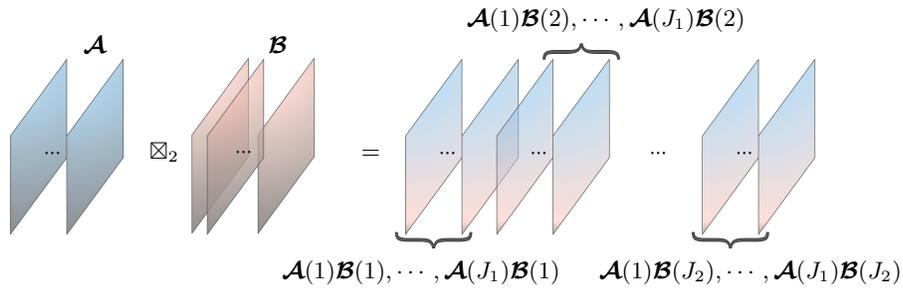}
 \begin{picture}(0,0)
 \put(-295,90){\footnotesize $\tensor{A}$}
 \put(-225,90){\footnotesize $\tensor{B}$}
 
 \put(-220,5){\footnotesize $\tensor{A}(1)\tensor{B}(1),\cdots,\tensor{A}(J_1)\tensor{B}(1)$}
 \put(-150,100){\footnotesize $\tensor{A}(1)\tensor{B}(2),\cdots,\tensor{A}(J_1)\tensor{B}(2)$}
 \put(-100,5){\footnotesize $\tensor{A}(1)\tensor{B}(J_2),\cdots,\tensor{A}(J_1)\tensor{B}(J_2)$}
 
 \put(-270,50){\footnotesize $\boxtimes_2$}
 \put(-190,50){\footnotesize $=$}
 \end{picture}
 } 
 \quad 
 \subfloat[Illustration of the slices-Hadamard product given in \Cref{def:hada_product}.]{
 \label{fig:hada_product}
 \includegraphics[scale=0.5]{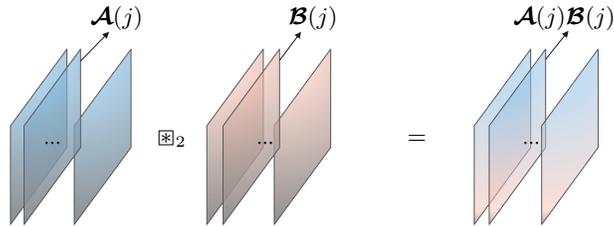}
 \begin{picture}(0,0)
 \put(-210,90){\footnotesize $\tensor{A}(j)$}
 \put(-135,90){\footnotesize $\tensor{B}(j)$}
 \put(-50,90){\footnotesize $\tensor{A}(j)\tensor{B}(j)$}

 \put(-185,45){\footnotesize $\boxast_2$}
 \put(-90,45){\footnotesize $=$}
 \end{picture}
 }
 \caption{Illustration of two new products of 3-order tensors.}
\end{figure}

Using this definition, we can rewrite the subchain tensor $\tensor{G}^{\ne n}$ in \cref{def:subchain} as follows 
\begin{equation}
\label{eq:subchain_new}
\tensor{G}^{\ne n} = \tensor{G}_{n+1} \boxtimes_2 \cdots \boxtimes_2 \tensor{G}_{N} \boxtimes_2 \tensor{G}_{1} \boxtimes_2 \cdots \boxtimes_2 \tensor{G}_{n-1}.
\end{equation}

In the following, we present a property of the above subchain product.
\begin{proposition}
	\label{prop:slice_kron}
	Let $\tensor{A} \in \bbR^{I_1 \times J_1 \times K}$ and $\tensor{B} \in \bbR^{K \times J_2 \times I_2}$ be two 3-order tensors, and $\bfA \in \bbR^{R_1 \times J_1}$ and $\bfB \in \bbR^{R_2 \times J_2}$ be two matrices. Then
	\begin{equation*}
	(\tensor{A} \times_2 \bfA) \boxtimes_2 (\tensor{B} \times_2 \bfB) = (\tensor{A} \boxtimes_2 \tensor{B}) \times_2 (\bfB \otimes \bfA).
	\end{equation*}
\end{proposition}

\begin{proof} 
    By \Cref{def:moden_product,def:subchain_product,def:kronecker}, upon some computations, we have
	\begin{subequations}
		\begin{align*}
		&((\tensor{A} \boxtimes_2 \tensor{B}) \times_2 (\bfB \otimes \bfA))(i_1,\overline{r_1r_2},i_2) \\
		&= \sum_{\overline{j_1j_2}=1}^{J_1J_2} (\tensor{A} \boxtimes_2 \tensor{B})(i_1,\overline{j_1j_2},i_2) (\bfB \otimes \bfA)(\overline{r_1r_2},\overline{j_1j_2}) &\text{by \Cref{def:moden_product}} \\
		&= \sum_{\overline{j_1j_2}=1}^{J_1J_2} \sum_{k=1}^{K} \tensor{A}(i_1,j_1,k)\tensor{B}(k,j_2,i_2) \bfB(r_2,j_2) \bfA(r_1,j_1) & \text{by \Cref{def:subchain_product,def:kronecker}} \\
		&= \sum_{k=1}^{K} \sum_{j_1=1}^{J_1} \tensor{A}(i_1,j_1,k) \bfA(r_1,j_1) \sum_{j_2=1}^{J_2}\tensor{B}(k,j_2,i_2)\bfB(r_2,j_2) \\ 
		&= \sum_{k=1}^{K} (\tensor{A} \times_2 \bfA)(i_1,r_1,k) (\tensor{B} \times_2 \bfB)(k,r_2,i_2) &\text{by \Cref{def:moden_product}} \\
		&= (\tensor{A} \times_2 \bfA) \boxtimes_2 (\tensor{B} \times_2 \bfB) (i_1,\overline{r_1r_2},i_2). & \text{by \Cref{def:subchain_product}}
		\end{align*}
	\end{subequations}
	Thus, the proof is completed.
\end{proof}

Based on the above findings, we can compute $\mcS \bfG^{\ne n}_{[2]}$ without forming $\mcS$ and $\bfG^{\ne n}_{[2]}$ explicitly by choosing a special $\mcS$. The details are presented in \Cref{sec:srft_tr}.

Furthermore, to facilitate the design of the algorithm by TensorSketch as in \cite{ma2021FastAccurate}, we define the following product, whose graphical
illustration is presented in  \Cref{fig:hada_product}.

\begin{definition} [Slices-Hadamard product]
	\label{def:hada_product}
	Let $\tensor{A} \in \bbR^{I_1 \times J \times K}$ and $\tensor{B} \in \bbR^{K \times J \times I_2}$ be two 3-order tensors, and $\bfA(j)$ and $\bfB(j)$ are the $j$-th lateral slices of $\tensor{A}$ and $\tensor{B}$, respectively. The mode-2 \textbf{slices-Hadamard product} of $\tensor{A}$ and $\tensor{B}$ is a tensor of size $I_1 \times J \times I_2$ denoted by $\tensor{A} \boxast_2 \tensor{B}$ and defined as
	\begin{equation*}
	(\tensor{A} \boxast_2 \tensor{B})(j) = \tensor{A}(j)\tensor{B}(j).
	\end{equation*}
	That is, the $j$-th lateral slice of $\tensor{A} \boxast_2 \tensor{B}$ is the classical matrix product of the $j$-th lateral slices of $\tensor{A}$ and $\tensor{B}$. The mode-1 and mode-3 slices-Hadamard product can be defined similarly.
\end{definition}

Using \Cref{def:kronecker}, similar to the proof of \Cref{prop:slice_kron}, we can prove the following result.
\begin{proposition}
	\label{prop:slice_hada}
	Let $\tensor{A} \in \bbR^{I_1 \times J_1 \times K}$ and $\tensor{B} \in \bbR^{K \times J_2 \times I_2}$ be two 3-order tensors, and $\bfA \in \bbR^{M \times J_1}$ and $\bfB \in \bbR^{M \times J_2}$ be two matrices. Then
	\begin{equation*}
	(\tensor{A} \times_2 \bfA) \boxast_2 (\tensor{B} \times_2 \bfB) = (\tensor{A} \boxtimes_2 \tensor{B}) \times_2 (\bfB^\intercal \odot \bfA^\intercal)^\intercal.
	\end{equation*}
\end{proposition}

With the slices-Hadamard product, we have the following formula which can avoid forming the TensorSketch and implementing the matrix multiplication between large matrices.
\begin{proposition}
	\label{prop:ts_tr}
	Let $\bfS_n = \mat{\Omega}_n \bfD_n \in \bbR^{m \times I_n}$, where $\mat{\Omega}_n \in \bbR^{m \times I_n}$ and $\bfD_n \in \bbR^{I_n \times I_n}$ are defined based on $H_n$ and  $S_n$ in \Cref{def:tensersketch_tr}, respectively. Let $\bfT \in  \bbR^{m \times \prod_{j=1}^N I_N}$ be defined in \Cref{def:tensersketch_tr} and $\tensor{P} = \tensor{A}^{(1)}\boxtimes_2 \tensor{A}^{(2)}\boxtimes_2 \cdots \boxtimes_2 \tensor{A}^{(N)}$ with $\tensor{A}^{(n)} \in \bbR^{R_n \times I_n \times R_{n+1}}$ for $n \in [N]$. Then
	\begin{equation*}
	\label{eq:tensorsketch}
	\tensor{P} \times_2 \bfT = \FFT^{-1}\left(
	\boxast_{2~n=1}^{~~N}
	\FFT \left(\tensor{A}^{(n)} \times_2 \mat{S}_{n}, [~], 2\right)
	, [~], 2 \right).
	\end{equation*}
\end{proposition}

\begin{proof}
	For simplicity, we only present the proof for $N=2$ here. Extending the proof to the general case is immediate. 
	
	On the one hand, the mode-$2$ products of $\tensor{A}^{(1)}$ with $\bfS_1$ and of $\tensor{A}^{(2)}$ with $\bfS_2$ can be represented by the following two $m-1$ degree matrix polynomials
	\begin{align*}
	\mcP^{(1)}(\omega)& = \sum_{i_1 = 1}^{I_1} S_1(i_1) \tensor{A}^{(1)}(:,i_1,:) \omega^{H_1(i_1)-1} = \sum_{m'=1}^{m} c^{(1)}_{m'} \omega^{m'-1} \in \bbR^{R_1 \times R_2},\\
	\mcP^{(2)}(\omega)& = \sum_{i_2 = 1}^{I_2} S_2(i_2) \tensor{A}^{(2)}(:,i_2,:) \omega^{H_2(i_2)-1} = \sum_{m'=1}^{m} c^{(2)}_{m'} \omega^{m'-1} \in \bbR^{R_2 \times R_3},
	\end{align*}
	where $c^{(1)}_{1},\ldots,c^{(1)}_{m}\in \bbR^{R_1 \times R_2}$ and $c^{(2)}_{1},\ldots,c^{(2)}_{m}\in \bbR^{R_2 \times R_3}$ are coefficient matrices, and $\omega$ satisfies $\omega^m=1$ but $\omega^{m'}\neq 1$ for $m'=1, \ldots, m-1$.  
	
	On the other hand, $\tensor{P} \times_2 \mat{T}$ can be represented as the following matrix polynomial
	\begin{align}
	\label{eq:poly_prod}
	\mcP(\omega)
	&= \sum_{i = 1}^{I_1I_2} S(i) \tensor{P}(:,i,:) \omega^{H(i)-1} = \sum_{m'=1}^{m} c_{m'} \omega^{m'-1} \in \bbR^{R_1 \times R_3} ,
	\end{align}
	where $c_{1},\ldots,c_{m}\in \bbR^{R_1 \times R_3}$ are coefficient matrices, $i = \overline{i_1i_2}$ and $\omega$ satisfies the constraints mentioned above.
	
	Note that 
	$\tensor{P} = \tensor{A}^{(1)}\boxtimes_2 \tensor{A}^{(2)} \in  \bbR^{R_1 \times I_1I_2 \times R_{3}}$ with $\tensor{A}^{(n)} \in \bbR^{R_n \times I_n \times R_{n+1}}$ for $n = 1,2$. Then, according to \Cref{def:subchain_product}, we have
	\begin{equation}
	\label{eq:tensorP}
	\tensor{P}(:,i,:) = \tensor{A}^{(1)}(:,i_1,:) \tensor{A}^{(2)}(:,i_2,:),
	\end{equation}
	where $i \in [I_1I_2]$, $i_1 \in [I_1]$, and $i_2 \in [I_2]$. Substituting \cref{eq:tensorP} into \cref{eq:poly_prod} implies 
	\begin{align*}
	\mcP(\omega)&= \sum_{\overline{i_1i_2} = 1}^{I_1I_2} S(\overline{i_1i_2}) \tensor{A}^{(1)}(:,i_1,:) \tensor{A}^{(2)}(:,i_2,:) \omega^{H(\overline{i_1i_2})-1}, 
	\end{align*}
	which together with \Cref{def:tensersketch_tr} and the constraints on $\omega$ mentioned above  gives
	\begin{align}
	\label{eq:poly_prod1}
	\nonumber\mcP(\omega)&= \sum_{i_1 = 1}^{I_1} \sum_{i_2 = 1}^{I_2} S_1(i_1)S_2(i_2) \tensor{A}^{(1)}(:,i_1,:) \tensor{A}^{(2)}(:,i_2,:)  \omega^{H_1(i_1)+H_2(i_2)-2~mod~m} \\
	\nonumber &=  \sum_{i_1 = 1}^{I_1} S_1(i_1) \tensor{A}^{(1)}(:,i_1,:) \omega^{H_1(i_1)-1} \sum_{i_2 = 1}^{I_2} S_2(i_2) \tensor{A}^{(2)}(:,i_2,:) \omega^{H_2(i_2)-1} \\
	&= \mcP^{(1)}(\omega) \mcP^{(2)}(\omega).
	\end{align}
	Thus, viewing  $$c_1,\cdots, c_m; ~\ c_1^{(1)},\cdots, c_m^{(1)}; ~\ c_1^{(2)},\cdots, c_m^{(2)}$$ as the lateral slices of the tensors $\tensor{C}, \tensor{C}^{(1)}$, and $\tensor{C}^{(2)}$, respectively, according to the famous convolution theorem, we can see that \cref{eq:poly_prod1} can be represented as 
	\begin{equation*}
	\tensor{C}= \FFT^{-1}\left( \FFT(\tensor{C}^{(1)},[~],2) \boxast_{2} \FFT(\tensor{C}^{(2)},[~],2), [~], 2 \right).
	\end{equation*}
	More precisely, 
	we have 
	\begin{equation*}
	\tensor{P} \times_2 \bfT = \FFT^{-1}\left(
	\FFT \left(\tensor{A}^{(1)} \times_2 \mat{S}_{1}, [~], 2\right)
	\boxast_{2}
	\FFT \left(\tensor{A}^{(2)} \times_2 \mat{S}_{2}, [~], 2\right)
	, [~], 2 \right).
	\end{equation*}
	This is the desired result.
\end{proof}

\section{KSRFT-based randomized TR decomposition}
\label{sec:srft_tr}
In this section, we present the randomized algorithm of TR decomposition based on KSRFT and discuss its theoretical analysis and computational complexity.

\subsection{Algorithms}
According to the analyses in \Cref{sec:main} and inspired by \cite{battaglino2018PracticalRandomized}, we choose $\mcS \in \bbR^{m \times \prod_{j \ne n} I_j}$ as\footnote{From \cref{def:SRFT}, there should be a scaling factor, i.e., $\sqrt{\frac{\prod_{j \ne n} I_j}{m} }$, in the right side of \cref{eq:s}. However, it has no effect on the solution. So we omit it here. In addition, there should also be a subscript, e.g., $\ne n$, in the matrices $\mcS, \bfS, \bfF$, and $\bfD$. For simplicity, we also omit it here. } 
\begin{equation}
\label{eq:s}
\mcS = \bfS \bfF \bfD,
\end{equation}
where
$\bfF =  \bigotimes_{\substack{n-1,\cdots,1,\\ N,\cdots,n+1}} \bfF_j $ and $\bfD =  \bigotimes_{\substack{n-1,\cdots,1,\\ N,\cdots,n+1}} \bfD_j $ with $\bfF_j\in \bbR^{I_j \times  I_j}$ and $\bfD_j\in \bbR^{I_j \times  I_j}$ being defined in \cref{def:SRFT}, and the meaning of $\bfS$ will be clarified later. Thus, the reduced problem of \cref{eq:tr_als} is

\begin{equation}
\label{eq:tr_fdals2}
\mathop{\arg\min}_{\bfG_{n(2)}} \left\| \bfS \bfF \bfD  \bfG_{[2]}^{\ne n} \bfG_{n(2)}^\intercal-\bfS \bfF \bfD \bfX_{[n]}^\intercal \right\|_F.
\end{equation}

From the discussions in \Cref{sec:main}, we know that finding $\bfF \bfD \bfG_{[2]}^{\ne n}$ can be transformed into finding $\tensor{G}^{\ne n} \times_2 (\bfF \bfD)$. While, using \cref{eq:subchain_new} and \Cref{prop:slice_kron}, we have
\begin{align}
\label{eq:subchain_fd}
\hat{\tensor{G}}^{\neq n} &= \tensor{G}^{\neq n} \times_2 (\bfF \bfD)  = \tensor{G}^{\neq n} \times_2 \left( \left( \bigotimes_{\substack{n-1,\cdots,1,\\ N,\cdots,n+1}} \bfF_j \right) \left( \bigotimes_{\substack{n-1,\cdots,1,\\ N,\cdots,n+1}} \bfD_j \right) \right) \\	
&\nonumber  = (\tensor{G}_{n+1} \times_2 (\bfF_{n+1} \bfD_{n+1})) \boxtimes_2 \cdots \boxtimes_2 (\tensor{G}_{N} \times_2 (\bfF_{N} \bfD_{N})) \\
&\nonumber \quad \quad \boxtimes_2 (\tensor{G}_{1} \times_2 (\bfF_{1} \bfD_{1})) \boxtimes_2 \cdots \boxtimes_2 (\tensor{G}_{n-1} \times_2 (\bfF_{n-1} \bfD_{n-1})).
\end{align}
As a result, $$\bfF \bfD \bfG_{[2]}^{\ne n} =  \hat \bfG_{[2]}^{\ne n}.$$ 
Thus, by mixing the lateral slices of each core $\tensor{G}_{j}$ via the small matrices $\bfF_j$ and $\bfD_j$, we get $\bfF \bfD \bfG_{[2]}^{\ne n}$, which can avoid forming the large matrices $\bfF$,  $\bfD$, and $\bfG^{\ne n}_{[2]}$ explicitly. 

Note that the second term in \cref{eq:tr_fdals2} is equivalent to
\begin{equation*}
\bfS \hat{\bfX}_{[n]}^\intercal (\bfD_n \bfF_n^*)^\intercal,
\end{equation*}
where $\hat{\tensor{X}}=\tensor{X} \times_1 (\bfF_1 \bfD_1) \times_2 (\bfF_2 \bfD_2) \cdots \times_N (\bfF_N \bfD_N)$, i.e., $\hat \bfX_{[n]} = \bfF_n \bfD_n \bfX_{[n]} ((\bfF_{n-1}\bfD_{n-1}) \otimes \cdots \otimes (\bfF_{1}\bfD_{1}) \otimes (\bfF_{N}\bfD_{N}) \otimes \cdots \otimes (\bfF_{n+1}\bfD_{n+1}))^ \intercal$, and the superscript ``$*$'' denotes the conjugate transpose. As explained in \cite{battaglino2018PracticalRandomized}, using the within the loop body and hence can significantly improve the computational efficiency.

Thus, our reduced problem has the following form
\begin{equation}
\label{eq:tr_argminsfd}
\mathop{\arg\min}_{\bfG_{n(2)}} \| \left(\bfS \hat{\bfG}_{[2]}^{\ne n} \right) \bfG_{n(2)}^\intercal- \bfS \hat{\bfX}_{[n]}^\intercal \left(\bfD_n \bfF_n^*\right)^\intercal \|_F.
\end{equation}

To avoid sampling the rows of $\hat \bfG^{\ne n}_{[2]}$, i.e., the lateral slices of the large mixed subchain tensor $\hat{\tensor{G}}^{\ne n}$ directly, we use the sampling strategy in \Cref{alg:sst}. That is, we sample the lateral slices of the $N-1$ cores individually and then form the sampled mixed subchain tensor $\hat{\tensor{G}}^{\ne n} \times_2 \bfS$. This process can be  expressed as follows:
\begin{align}
\label{eq:subchain_sketch}
\hat{\tensor{G}}^{\ne n} \times_2 \bfS &= ( (\tensor{G}_{n+1} \times_2 (\bfF_{n+1} \bfD_{n+1})) \boxtimes_2 
\cdots \boxtimes_2 (\tensor{G}_{N} \times_2 (\bfF_{N} \bfD_{N})) \\
&\nonumber \quad\quad \boxtimes_2 (\tensor{G}_{1} \times_2 (\bfF_{1} \bfD_{1})) \boxtimes_2 \cdots \boxtimes_2 (\tensor{G}_{n-1} \times_2 (\bfF_{n-1} \bfD_{n-1})) ) \times_2 \bfS \\
&\nonumber = (\tensor{G}_{n+1} \times_2 (\bfS_{n+1} \bfF_{n+1} \bfD_{n+1})) \boxast_2 
\cdots \boxast_2 (\tensor{G}_{N} \times_2 (\bfS_{N} \bfF_{N} \bfD_{N})) \\ 
&\quad \quad \boxast_2 (\tensor{G}_{1} \times_2 (\bfS_{1} \bfF_{1} \bfD_{1})) \boxast_2 \cdots \boxast_2 (\tensor{G}_{n-1} \times_2 (\bfS_{n-1} \bfF_{n-1} \bfD_{n-1})),\nonumber
\end{align}
where $\bfS_j\in \bbR^{m \times  I_j}$ is the sampling matrix for the $j$-th core. According to \Cref{prop:slice_hada}, we can see that $\bfS = \left( \bigodot_{\substack{n-1,\cdots,1,\\ N,\cdots,n+1}} \bfS_j^\intercal \right)^\intercal$. Thus, using the property of Khatri-Rao product, we have
\begin{equation}
\label{eq:sketch_ksrft}
	\mcS = \left( \bigodot_{\substack{n-1,\cdots,1,\\ N,\cdots,n+1}} \bfS_j^\intercal \right)^\intercal \left( \bigotimes_{\substack{n-1,\cdots,1,\\ N,\cdots,n+1}} (\bfF_j \bfD_j)\right)= \left( \bigodot_{\substack{n-1,\cdots,1,\\ N,\cdots,n+1}} (\bfS_j\bfF_j\bfD_j)^\intercal \right)^\intercal,
\end{equation}
and hence we can also call it the Khatri-Rao SRFT. Note that the $\mcS$ used in \cite{battaglino2018PracticalRandomized} has the same structure as $\mcS$ in \cref{eq:sketch_ksrft} though the structure was not provided explicitly there. 

The above discussions imply  \Cref{alg:tr_srft_als}.

\begin{algorithm}
	\caption{TR-KSRFT-ALS (Proposal)}
	\label{alg:tr_srft_als}
	\begin{algorithmic}[1]\footnotesize
		\Function{$\{\tensor{G}_n\}_{n=1}^N$= TR-KSRFT-ALS}{$\tensor{X}, R_1, \cdots, R_N, m$} 
		
		\Comment $\tensor{G}_n \in \bbR^{R_n \times I_n \times R_{n+1}}, n \in [N]$; $\tensor{X} \in \bbR^{I_1 \times \cdots \times I_N}$ 
		
		\Comment $R_1, \cdots, R_N$ are the TR-ranks
		
		\Comment $m$ is the uniform sampling size 
		\State Initialize cores ${\tensor{G}}_2, \cdots, {\tensor{G}}_N$ \label{line:srft_init}
		\State Define random sign-flip operators $\bfD_j$ and FFT matrices $\bfF_j$, for $j \in [N]$
		\State Mix cores: $\hat{\tensor{G}}_{n} \leftarrow \tensor{G}_{n} \times_2 (\bfF_n \bfD_n)$, for $ n = {2,\cdots, N}$ \label{line:cores_mix}
		\State Mix tensor: $\hat{\tensor{X}} \leftarrow \tensor{X} \times_1 (\bfF_1 \bfD_1) \times_2 (\bfF_2 \bfD_2) \cdots \times_N (\bfF_N \bfD_N)$ \label{line:srft_mix}
		\Repeat
		\For{$n = 1, \cdots, N$}
		\State Define sampling operator $\bfS \in \bbR^{m \times \prod_{j \ne n}I_j}$ \label{line:srft_S}
		\State Retrieve \texttt{idxs} from $\bfS$ \label{line:srft_idx}
		\State $\hat{\tensor{G}}^{\ne n}_S$ = SST($\texttt{idxs}, \hat{\tensor{G}}_{n+1}, \cdots, \hat{\tensor{G}}_{N}, \hat{\tensor{G}}_{1}, \cdots, \hat{\tensor{G}}_{n-1}$) \label{line:srft_sst}
		\State $\hat{\bfX}_{S[n]}^\intercal \leftarrow \bfS \hat{\bfX}_{[n]}^\intercal \left(\bfD_n \bfF_n^*\right)^\intercal$ \label{line:srft_SX}
		\State Update $\tensor{G}_n = \mathop{\arg\min}_{\tensor{Z}} \| \hat{\bfG}_{S[2]}^{\ne n} \bfZ_{(2)}^\intercal - \hat{\bfX}_{S[n]}^\intercal \|_F$ subject to $\tensor{G}_n$ being real-valued \label{line:srft_ls}
		\State $\hat{\tensor{G}}_{n} \leftarrow \tensor{G}_{n} \times_2 (\bfF_n \bfD_n)$ \label{line:srft_remix}
		\EndFor
		\Until{termination criteria met}
		\State \Return  $\tensor{G}_1, \cdots, \tensor{G}_N$
		\EndFunction
	\end{algorithmic}
\end{algorithm}

Furthermore, note that, with the property of the Frobenius norm, the reduced problem \cref{eq:tr_argminsfd} can be rewritten as 
\begin{equation*}
\mathop{\arg\min}_{\bfG_{n(2)}} \| \left(\bfS \hat{\bfG}_{[2]}^{\ne n} \right) \bfG_{n(2)}^\intercal \left(\bfF_n \bfD_n \right)^\intercal - \bfS \hat{\bfX}_{[n]}^\intercal  \|_F.
\end{equation*}
Thus, setting $\hat \bfG_{n(2)} = \bfF_n \bfD_n \bfG_{n(2)}$, we have 
\begin{equation}\label{eq:tr_argminsfd2}
\mathop{\arg\min}_{\hat \bfG_{n(2)}} \| \left(\bfS \hat{\bfG}_{[2]}^{\ne n} \right) \hat \bfG_{n(2)}^\intercal- \left(\bfS \hat{\bfX}_{[n]}^\intercal \right) \|_F.
\end{equation}
Consequently, we can solve the above problem to get $\hat{\tensor{G}}_n$ first and then recover the original core $\tensor{G}_n$ from it. The specific algorithm is summarized in \Cref{alg:pre_tr_srft_als}.

\begin{algorithm}
	\caption{TR-KSRFT-ALS-Premix (Proposal)}
	\label{alg:pre_tr_srft_als}
	\begin{algorithmic}[1]\footnotesize
		\Function{$\{\tensor{G}_n\}_{n=1}^N$= TR-KSRFT-ALS-Premix}{$\tensor{X}, R_1, \cdots, R_N, m$} 
		
		\Comment $\tensor{G}_n \in \bbC^{R_n \times I_n \times R_{n+1}}, n\in[N]$; $\tensor{X} \in \bbC^{I_1 \times \cdots \times I_N}$ 
		
		\Comment $R_1, \cdots, R_N$ are the TR-ranks
		
		\Comment $m$ is the uniform sampling size

		\State Define random sign-flip operators $\bfD_j$ and FFT matrices $\bfF_j$, for $j \in [N]$
		\State Mix tensor: $\hat{\tensor{X}} \leftarrow \tensor{X} \times_1 (\bfF_1 \bfD_1) \times_2 (\bfF_2 \bfD_2) \cdots \times_N (\bfF_N \bfD_N)$ \label{line:p_srft_mix}
		\State Initialize cores $\hat{\tensor{G}}_2, \cdots, \hat{\tensor{G}}_N$ \label{line:p_srft_init}
		\Repeat
		\For{$n = 1, \cdots, N$}
		\State Define sampling operator $\bfS \in \bbR^{m \times \prod_{j \ne n}I_j}$ \label{line:p_srft_S}
		\State Retrieve \texttt{idxs} from $\bfS$ \label{line:p_srft_idx}
		\State $\hat{\tensor{G}}^{\ne n}_S$ = SST($\texttt{idxs}, \hat{\tensor{G}}_{n+1}, \cdots, \hat{\tensor{G}}_{N}, \hat{\tensor{G}}_{1}, \cdots, \hat{\tensor{G}}_{n-1}$) \label{line:p_srft_sst}
		\State $\hat{\bfX}_{S[n]}^\intercal \leftarrow \bfS \hat{\bfX}_{[n]}^\intercal$ \label{line:p_srft_SX}
		\State Update $\hat{\tensor{G}}_n = \mathop{\arg\min}_{\tensor{Z}} \| \hat{\bfG}_{S[2]}^{\ne n} \bfZ_{(2)}^\intercal - \hat{\bfX}_{S[n]}^\intercal \|_F$ \label{line:p_srft_ls}
		\EndFor
		\Until{termination criteria met}
		\For{$n = 1, \cdots, N$}
		\State Unmix cores: $\tensor{G}_{n} \leftarrow \hat{\tensor{G}}_{n} \times_2 (\bfD_n \bfF_n^*)$ \label{line:p_srft_unmix}
		\EndFor
		\State \Return  $\tensor{G}_1, \cdots, \tensor{G}_N$
		\EndFunction
	\end{algorithmic}
\end{algorithm}

\begin{remark}
	\label{rem:outlinerr}
	Algorithms \ref{alg:tr_srft_als} and \ref{alg:pre_tr_srft_als} are inspired by \cite{battaglino2018PracticalRandomized}, which designed some practical randomized algorithms for CP decomposition. One of our main contributions is that we present an elegant way to transform the problem \cref{eq:tr_als} to the problem \cref{eq:tr_argminsfd} with the new tensor product and its property. This allows us to design algorithms as elegantly as in \cite{battaglino2018PracticalRandomized}. Compared with the method in \cite{malik2020SamplingBased}, which is mainly inspired by the work in \cite{larsen2020PracticalLeverageBased}, our method may work better for some special data, such as for the data whose TR-cores include outliers\footnote{It is obviously the posterior assumption because we do not know the true TR-cores of a tensor in advance in practice. This is similar to what is discussed on the algorithms of CP  decomposition in \cite[pp. 890-891]{battaglino2018PracticalRandomized} and \cite[p. 33]{vandecappelle2021NumericalAlgorithms}, where the authors consider the problem with correlated factor matrix columns. \label{foot:outlier} } \cite{gong2020RobustGradientBased,gong2021markov}. This is mainly because, with the sampling method in \Cref{alg:sst}, TR-ALS-Sampled may sample the aforementioned outliers with high probability many times and hence may cause an unstable result. Although we also use \Cref{alg:sst}, the sampling in our method is uniform sampling and the tensor is transformed by the randomized Fourier transform before sampling. The above analysis is supported by the numerical results of the third experiment on synthetic data in \Cref{sec:numerical}.
\end{remark}

\begin{remark}
	\label{rem:equivalent}
	Upon close examination of the derivation of \Cref{alg:tr_srft_als}, we may let $\bfF_j \bfD_j$ be any suitable randomized matrices. For example, we can set them to be the CountSketch matrices for sparse input data. The corresponding method can be denoted by TR-CS-ALS. Moreover, we are surprised to find that if we set them to be the transpose of the factor matrices appearing in the randomized Tucker decomposition, we can recover the algorithm rTR-ALS proposed in \cite{yuan2019RandomizedTensor}. 
	These analyses present an affirmative result that reducing the dimension of each mode of the original tensor with a suitable sketching matrix, finding the TR decomposition of the reduced tensor, and then recovering the TR decomposition of the original tensor is feasible. Note that, in these cases, the step on sampling is not necessary and hence the final product is the subchain product as in \cref{eq:subchain_fd} but not the slices-Hadamard product as in \cref{eq:subchain_sketch}.
	The differences are illustrated in \Cref{fig:all}, from which we can also find that the sketching-based methods for TR decomposition can be mainly classified into two classes: one is based on the subchain product, and the other is based on the slices-Hadamard product. We also summarize all the algorithms mentioned above in \Cref{tab:compare_alg}.
	
	\begin{table}[htbp]
	\caption{Comparisons of randomized algorithms for TR decomposition}
	\label{tab:compare_alg}
		\resizebox{1\linewidth}{!}{
		\begin{tabular}{llll}
		\hline
		\textbf{Algorithms}                                                               & $\bfS_j$ \textbf{in \Cref{fig:all}}                                                                                                                  & $\mcS$ \textbf{in \Cref{fig:all}}                                                                                     & \textbf{Product}                                                           \\ \hline
		rTR-ALS (\cite{yuan2019RandomizedTensor})                                & \begin{tabular}[c]{@{}l@{}}$\bfQ_j^\intercal$ from the QR decomposition of $\bfX_{[j]} \bfM$ \\ with $\bfM$ being a Gaussian matrix\end{tabular} & $ \bigotimes_{\substack{n-1,\cdots,1,\\ N,\cdots,n+1}} \bfQ_j^\intercal$                                   & Subchain product                                                   \\ \hline
		TR-ALS-Sampled (\cite{malik2020SamplingBased})                           & sampling matrix via leverage score                                                                                            & $\left( \bigodot_{\substack{n-1,\cdots,1,\\ N,\cdots,n+1}} (\bfS_j)^\intercal \right)^\intercal$             & \begin{tabular}[c]{@{}l@{}}Slices-Hadamard \\ product\end{tabular} \\ \hline
		\begin{tabular}[c]{@{}l@{}}TR-KSRFT-ALS \\ (Proposal)\end{tabular}       & $\bfS_j \bfF_j \bfD_j$                                                                                                      & $\left( \bigodot_{\substack{n-1,\cdots,1,\\ N,\cdots,n+1}} (\bfS_j\bfF_j\bfD_j)^\intercal \right)^\intercal$ & \begin{tabular}[c]{@{}l@{}}Slices-Hadamard \\ product\end{tabular} \\ \hline
		\begin{tabular}[c]{@{}l@{}}TR-TS-ALS \\ (Proposal)\end{tabular}          &  CountSketch                                                                                                            & \begin{tabular}[c]{@{}l@{}} TensorSketch \end{tabular}                  & \begin{tabular}[c]{@{}l@{}}Slices-Hadamard \\ product\end{tabular} \\ \hline
		\begin{tabular}[c]{@{}l@{}}TR-KSRFT-ALS \\ without sampling\end{tabular} & $\bfF_j \bfD_j$                                                                                                           & $ \bigotimes_{\substack{n-1,\cdots,1,\\ N,\cdots,n+1}} \bfF_j  \bfD_j$                                     & Subchain product                                                   \\ \hline
		\begin{tabular}[c]{@{}l@{}} TR-CS-ALS\end{tabular}     &  CountSketch                                                                                                          & $ \bigotimes_{\substack{n-1,\cdots,1,\\ N,\cdots,n+1}} \bfS_j$                                             & Subchain product                                                   \\ \hline
		\end{tabular}
		}
	\end{table}
	
	\begin{figure}[htbp]
		\centering  
		\includegraphics[width=1\columnwidth]{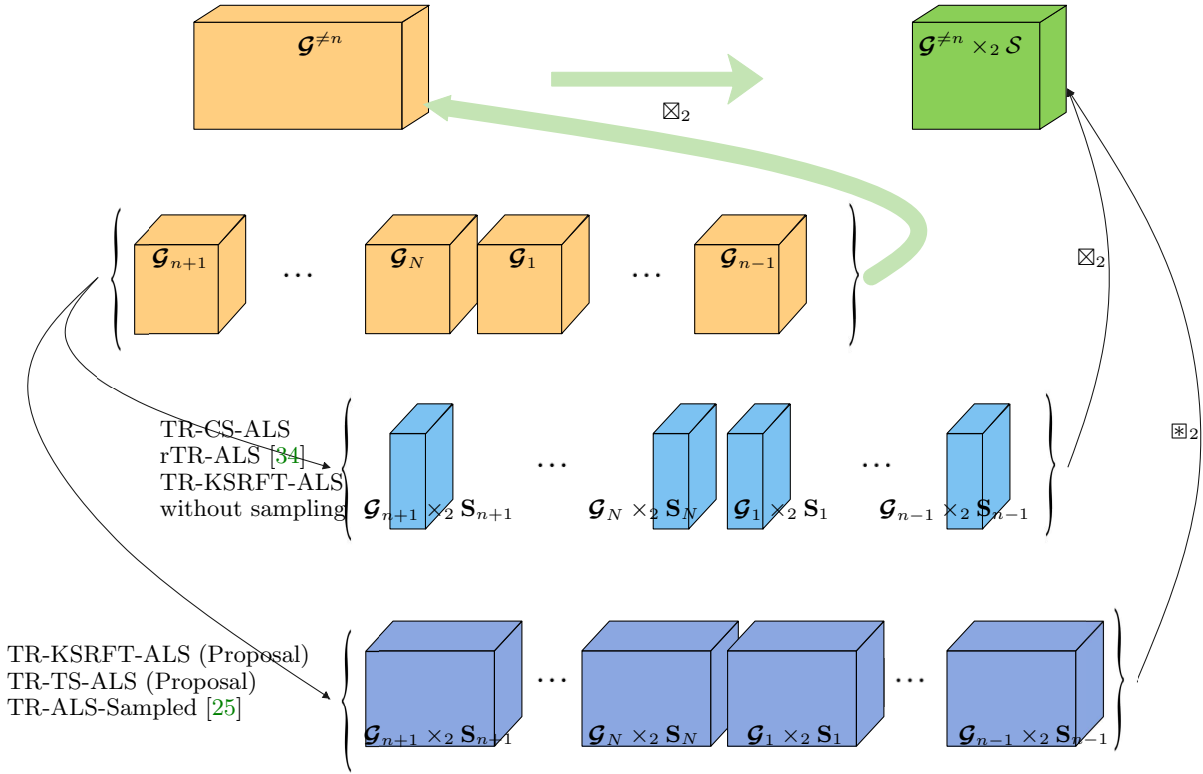}
		\caption{Illustration of how to efficiently construct $\tensor{G}^{\ne n} \times_2 \mcS$ by sketching the core tensors.}
		\label{fig:all}
		\begin{picture}(0,0)
		\put(-120,320){\footnotesize $\tensor{G}^{\ne n}$}
		\put(115,320){\footnotesize  $\tensor{G}^{\ne n} \times_2 \mcS$}
		
		\put(18,295){{\footnotesize $\boxtimes_2$}}
		
		\put(-175,240){{\footnotesize $\tensor{G}_{n+1}$}}
		\put(-85,240){{\footnotesize $\tensor{G}_{N}$}}
		\put(-40,240){{\footnotesize $\tensor{G}_{1}$}}
		\put(40,240){{\footnotesize $\tensor{G}_{n-1}$}}
		
		\put(-172,175){{\footnotesize TR-CS-ALS}}
		\put(-172,165){{\footnotesize rTR-ALS \cite{yuan2019RandomizedTensor}}}
		\put(-172,155){{\footnotesize TR-KSRFT-ALS}}
		\put(-172,145){{\footnotesize without sampling}}
		
		\put(-230,90){{\footnotesize TR-KSRFT-ALS (Proposal)}}
		\put(-230,80){{\footnotesize TR-TS-ALS (Proposal)}}
		\put(-230,70){{\footnotesize TR-ALS-Sampled \cite{malik2020SamplingBased}}}
		
		\put(-95,145){{\footnotesize $\tensor{G}_{n+1} \times_2 \bfS_{n+1}$}}
		\put(-10,145){{\footnotesize $\tensor{G}_{N} \times_2 \bfS_{N}$}}
		\put(45,145){{\footnotesize $\tensor{G}_{1} \times_2 \bfS_{1}$}}
		\put(100,145){{\footnotesize $\tensor{G}_{n-1} \times_2 \bfS_{n-1}$}}

		\put(-95,60){{\footnotesize $\tensor{G}_{n+1} \times_2 \bfS_{n+1}$}}
		\put(-10,60){{\footnotesize $\tensor{G}_{N} \times_2 \bfS_{N}$}}
		\put(50,60){{\footnotesize $\tensor{G}_{1} \times_2 \bfS_{1}$}}
		\put(130,60){{\footnotesize $\tensor{G}_{n-1} \times_2 \bfS_{n-1}$}}

		\put(175,240){{\footnotesize $\boxtimes_2$}}
		\put(210,175){{\footnotesize $\boxast_2$}}
		
		\end{picture}
	\end{figure} 
\end{remark}

\begin{remark}
	\label{rem:complex}
	As done in \cite{battaglino2018PracticalRandomized}, to get the real-valued TR-cores from the real-valued input tensor, we should solve  the following equivalent form of \LineRef{line:srft_ls} in \Cref{alg:tr_srft_als}:
	\begin{equation*}
	\tensor{G}_n = \mathop{\arg\min}_{\tensor{Z}} \left\|\begin{bmatrix} \mfR(\hat{\bfG}_{S[2]}^{\ne n}) \\
	\mfI(\hat{\bfG}_{S[2]}^{\ne n}) \end{bmatrix} \bfZ_{(2)}^\intercal - \begin{bmatrix} \mfR(\hat{\bfX}_{S[n]}^\intercal) \\ \mfI(\hat{\bfX}_{S[n]}^\intercal) \end{bmatrix} \right\| _F. 
	\end{equation*}
	Of course, if the matrices $\bfF_j$ used in the algorithm are real-valued orthogonal, this process can be ignored.
	For \Cref{alg:pre_tr_srft_als}, it is very applicable to finding the complex-valued core tensors from the complex-valued input tensor. 
\end{remark}

\subsection{Theoretical analysis}
Since the problem we focus on is essentially a sketched least squares problem like the one in \cite{battaglino2018PracticalRandomized,jin2021FasterJohnson}, we can apply Proposition 2.1 in \cite{jin2021FasterJohnson} to our method. The results are given as follows.

\begin{theorem}[A slight restatement of Proposition 2.1 in \cite{jin2021FasterJohnson}]
\label{thm:tr_als_ksrft}
	For the matrices $\bfG_{[2]}^{\ne n}$ and $\bfX_{[n]}^\intercal$ in \cref{eq:tr_als}, denote $rank(\bfG_{[2]}^{\ne n}) = r$ and fix $\varepsilon, \eta \in (0,1)$ such that $\prod_{j \ne n} I_j \lesssim 1/\varepsilon^r$ with $r \geq 2$. Then a sketching matrix $\mcS \in \bbC^{m \times \prod_{j \ne n} I_j}$ used in Algorithms \ref{alg:tr_srft_als} or \ref{alg:pre_tr_srft_als}, i.e., \cref{eq:sketch_ksrft} with
	\begin{equation*}
	m = \bigO{\varepsilon^{-1}r^{2(N-1)} \log^{2N-3}(\frac{r}{\varepsilon}) \log^4(\frac{r}{\varepsilon} \log(\frac{r}{\varepsilon})) \log \prod_{j \ne n} I_j}
	\end{equation*}
	is sufficient to output
	\begin{equation*}
	\tilde{\bfG}_{n(2)}^\intercal = \mathop{\arg\min}_{\bfG_{n(2)}^\intercal \in \bbR^{R_n R_{n+1} \times I_n}} \|\mcS \bfG_{[2]}^{\ne n} \bfG_{n(2)}^\intercal-\mcS \bfX_{[n]}^\intercal \|_F,
	\end{equation*}
	such that
	{\small\begin{align*}
	\mathbf{Pr}  \left( \| \bfG_{[2]}^{\ne n} \tilde{\bfG}_{n(2)}^\intercal - \bfX_{[n]}^\intercal \|_F = (1 \pm \bigO{\varepsilon}) \mathop{\min} \| \bfG_{[2]}^{\ne n} \bfG_{n(2)}^\intercal- \bfX_{[n]}^\intercal \|_F \right) 
	&\geq 1-\eta-2^{-\Omega(\log \prod_{j \ne n} I_j)}.
	\end{align*}}
\end{theorem}

\begin{remark}
\label{rem:ksrft_nocover}
    It is worth clarifying that the algorithms  rTR-ALS,TR-KSRFT-ALS without sampling and TR-CS-ALS listed in \Cref{tab:compare_alg} aren't covered by the guarantees in the above theorem. This is because there is no step on sampling in these methods.  
\end{remark}

\subsection{Computational complexity}
\label{sec:srft_cost}
We first analyze the computational complexity of \Cref{alg:tr_als} because our algorithms are based on TR-ALS. For simplicity, we assume that $I_n = I$ and $R_n = R$ for all $n \in [N]$ and that $N < I, R^2< I$ and $NR < I$. And we ignore any cost associated with e.g. normalization and checking termination conditions.

Upfront costs of TR-ALS:
\begin{itemize}
	\item \textbf{\LineRef{line:als_init}: Initializing cores.} This depends on how to initiate the cores. We assume that they are randomly drawn, e.g. from a Gaussian distribution, resulting in a cost $\bigO{NIR^2}$.
\end{itemize}

The costs of per outer loop iteration in TR-ALS:
\begin{itemize}
	\item \textbf{\LineRef{line:als_subchain}: Computing the unfolding subchain tensor.} If the $N-1$ cores are dense and contracted in sequence, the cost is $$R^3(I^2+I^3+ \cdots +I^{N-1}) \leq R^3(NI^{N-2}+I^{N-1}) \leq 2R^3I^{N-1} = \bigO{I^{N-1}R^3}.$$
	Doing this for each of the $N$ cores in the inner loop brings the cost to $\bigO{NI^{N-1}R^3}$.
	\item \textbf{\LineRef{line:als_ls}: Solving the least squares problem.} We use the standard QR-based approach to analyze the computational complexity. Doing a QR decomposition of $\bfG^{\ne n}_{[2]} \in \bbR^{I^{N-1} \times R^2}$ costs $\bigO{I^{N-1}R^4}$, and updating the right hand side and doing back substitution costs $\bigO{I(I^{N-1}R^2+R^4)} = \bigO{I^N R^2}$. The leading order complexity for solving the least squares problem is therefore $\bigO{I^N R^2}$. Doing this for each of the N cores in the inner loop brings the cost to $\bigO{N I^N R^2}$.
\end{itemize}

Putting them all together, we have that the leading order complexity of TR-ALS is 
$$\bigO{NIR^2 + it \cdot N I^N R^2},$$ 
where ``$it$'' denotes the number of outer loop iterations.

Now, we present the computational complexity of \Cref{alg:tr_srft_als}.

Upfront costs of TR-KSRFT-ALS:
\begin{itemize}
	\item \textbf{\LineRef{line:srft_init}: Initializing cores.} It is the same as TR-ALS, which costs $\bigO{NIR^2}$.
	\item \textbf{\LineRef{line:cores_mix}:  Mixing cores.} It costs $\bigO{(N-1)I \log I}$ because there are $N-1$ cores.
	\item \textbf{\LineRef{line:srft_mix}: Mixing tensor.} 
	It requires a significant upfront cost, which is $\bigO{I^N \log(I^N)}$.
\end{itemize}

The costs of per outer loop iteration in TR-KSRFT-ALS:
\begin{itemize}
	\item \textbf{\LineRef{line:srft_S}: Generating $N$ sampling matrices.} 
	It costs $\bigO{mN}$. Actually, these matrices are not generated explicitly in specific implementation.
	\item \textbf{\LineRef{line:srft_idx}-\LineRef{line:srft_sst}: Computing the sampled subchain tensor.} The main cost is the matrix multiplication of $N-1$ matrices of size $R \times R$ for each of the $m$ sampled mode-2 slices, which costs $\bigO{mR^3 N}$ per inner loop iteration, i.e., $\bigO{mR^3 N^2}$ per outer loop iteration.
	\item  \textbf{\LineRef{line:srft_SX}: Sampling the input tensor.} It needs to copy $\bigO{mI}$ elements from the input tensor and then do matrix multiplication. So the cost is $\bigO{mI^2}$ per inner loop iteration, i.e., 
	$\bigO{NmI^2}$ per outer loop iteration.  
	\item \textbf{\LineRef{line:srft_ls}: Solving the least squares problem.} Similar to TR-ALS, it costs $\bigO{ImR^2}$ per inner loop iteration, i.e., $\bigO{NImR^2}$ per outer loop iteration.
	\item \textbf{\LineRef{line:srft_remix}: Update the mixed cores.} It costs $\bigO{I \log I}$ per inner loop iteration, i.e., $\bigO{NI \log I}$ per outer loop iteration.
\end{itemize}

It follows that the overall leading order complexity of TR-KSRFT-ALS is 
$$\bigO{I^N \log(I^N) + it \cdot NmI^2},$$ 
where ``$it$'' denotes the number of outer loop iterations.

For the computational complexity of TR-KSRFT-ALS-Premix, the analysis is similar. Next we only show the difference:
\begin{itemize}
	\item \textbf{\LineRef{line:p_srft_unmix}: Unmixing the cores.} It costs $\bigO{NI \log I}$ because there are $N$ cores.
\end{itemize}

So the overall leading order complexity of TR-KSRFT-ALS-Premix is $$\bigO{I^N \log(I^N) + it \cdot NImR^2},$$ 
where ``$it$'' denotes the number of outer loop iterations.
\begin{remark}
	Like Algorithms 4 and 5 in \cite{battaglino2018PracticalRandomized}, the dominant cost of TR-KSRFT-ALS and TR-KSRFT-ALS-Premix also appears in the upfront costs, i.e., the mixing tensor step. Without considering this cost, i.e., we assume that the data has been preprocessed, the complexities of TR-KSRFT-ALS and TR-KSRFT-ALS-Premix can be reduced remarkably when $I$ and $N$ are huge.
\end{remark}

\section{TS-based randomized TR decomposition}
\label{sec:ts_tr}
In this section, we set $\mcS$ to be the TensorSketch $\bfT_{\ne n} \in \bbR^{m \times \prod_{j \ne n}I_j}$ defined as in \cref{def:tensersketch_tr} with $j=1, \ldots, n-1, n+1, \ldots, N$.
Thus, from the discussions in \Cref{sec:main,sec:srft_tr}, especially \cref{eq:subchain_new} and  \Cref{prop:ts_tr}, we can propose the following \Cref{alg:tr_ts_als} for TR decomposition based on the problem \cref{eq:tr_als}, and the \LineRef{line:ts_TG} is illustrated in \Cref{fig:all} and \Cref{tab:compare_alg}.

\begin{algorithm}
	\caption{TR-TS-ALS (Proposal)}
	\label{alg:tr_ts_als}
	\begin{algorithmic}[1]\footnotesize
		\Function{$\{\tensor{G}_n\}_{n=1}^N$= TR-TS-ALS}{$\tensor{X}, R_1, \cdots, R_N, m$} 
		
		\Comment $\tensor{G}_n \in \bbR^{R_n \times I_n \times R_{n+1}}, n \in [N]$; $\tensor{X} \in \bbR^{I_1 \times \cdots \times I_N}$ 
		
		\Comment $R_1, \cdots, R_N$ are the TR-ranks
		
		\Comment $m$ is the embedding size
		
		\State Define $\bfS_j$, i.e., the CountSketch, based on $H_n$ and  $S_n$ in \Cref{def:tensersketch_tr}, for $j \in [N]$ \label{line:ts_S}
		\For{$n = 1, \cdots, N$}
		\State Compute the sketch of $\bfX_{[n]}^\intercal$: $\hat{\bfX}_{[n]}^\intercal \leftarrow \bfT_{\ne n} \bfX_{[n]}^\intercal$ \label{line:ts_TX}
		\EndFor
		\State Initialize cores $\tensor{G}_2, \cdots, \tensor{G}_N$ \label{line:ts_init}
		\Repeat
		\For{$n = 1, \cdots, N$}
		\State Compute $\hat{\tensor{G}}^{\neq n} 
		= \FFT^{-1}\left( \boxast_{2~j=n+1,\cdots, N}^{~~1,\cdots, n-1} \FFT \left(\tensor{G}_{j} \times_2 \mat{S}_{j}, [~], 2\right), [~], 2 \right)$ \label{line:ts_TG}
		\State Update $\tensor{G}_n = \mathop{\arg\min}_{\tensor{Z}} \| \hat{\bfG}_{[2]}^{\ne n} \bfZ_{(2)}^\intercal - \hat{\bfX}_{[n]}^\intercal \|_F$ \label{line:ts_ls}
		\EndFor
		\Until{termination criteria met}
		\State \Return  $\tensor{G}_1, \cdots, \tensor{G}_N$
		\EndFunction
	\end{algorithmic}
\end{algorithm}

\begin{remark}
	\Cref{alg:tr_ts_als} is inspired by the work in \cite{ma2021FastAccurate}, which proposes a fast and accurate sketched ALS algorithm for Tucker decomposition by using TensorSketch. Our main contribution is that we find the graceful formula in \cref{eq:tensorsketch} using the new tensor product and its property. With it, we can design the algorithm as efficiently as in \cite{ma2021FastAccurate}.
\end{remark}

Similar to Theorem A.7 in \cite{ma2021FastAccurate}, we have the following theoretical sketch size for achieving $\bigO{\varepsilon}$-relative error.
\begin{theorem}[A slight restatement of Theorem A.7 in \cite{ma2021FastAccurate}]
\label{thm:tr_als_ts}
	For the matrices $\bfG_{[2]}^{\ne n}$ and $\bfX_{[n]}^\intercal $ in \cref{eq:tr_als}, fix $\varepsilon, \eta \in (0,1)$. Then a TensorSketch $\bfT_{\ne n}$ used in \Cref{alg:tr_ts_als} with
	\begin{equation*}
	m = \bigO{((R_{n}R_{n+1} \cdot 3^{N-1})((R_{n}R_{n+1} + 1/ \varepsilon^2 )/ \eta},
	\end{equation*}
	is sufficient to output
	\begin{equation*}
	\tilde{\bfG}_{n(2)}^\intercal = \mathop{\arg\min}_{\bfG_{n(2)}^\intercal \in \bbR^{R_n R_{n+1} \times I_n}} \|\bfT_{\ne n} \bfG_{[2]}^{\ne n} \bfG_{n(2)}^\intercal-\bfT_{\ne n} \bfX_{[n]}^\intercal \|_F,
	\end{equation*}
	such that 
	\begin{equation*}
	\mathbf{Pr}  \left( \| \bfG_{[2]}^{\ne n} \tilde{\bfG}_{n(2)}^\intercal - \bfX_{[n]}^\intercal \|_F =  (1 \pm \bigO{\varepsilon}) \mathop{\min} \| \bfG_{[2]}^{\ne n} \bfG_{n(2)}^\intercal- \bfX_{[n]}^\intercal \|_F \right) \geq 1-\eta.
	\end{equation*}
\end{theorem}

Now we consider the compliexiy analysis of \Cref{alg:tr_ts_als}. Specifically, with the same assumptions as in \Cref{sec:srft_cost}, we have the following results.

\begin{itemize}
	\item \textbf{\LineRef{line:ts_S}: Generating the $N$ CountSketch matrices.} It costs $\bigO{NI}$.
	\item \textbf{\LineRef{line:ts_TX}: Computing the TensorSketch of the input tensor.} It costs $\bigO{Nnnz(\tensor{X})}$.
	\item \textbf{\LineRef{line:ts_init}: Initializing cores.} It is the same as TR-ALS, which costs $\bigO{NIR^2}$.
	\item \textbf{\LineRef{line:ts_TG}: Computing the TensorSketch of the subchain tensor.} It costs
	{\begin{align*}
		&(N-1)nnz(\tensor{G}_j)+(N-1)mR^2 \log(m)+(N-2)mR^3+mR^2 \log(m) \\
		&= \bigO{Nnnz(\tensor{G}_j)+NmR^2 \log(m)+NmR^3}
		\end{align*}}
	\item \textbf{\LineRef{line:ts_ls}: Solving the least squares problem.} Similar to TR-ALS, it costs $\bigO{ImR^2}$ per inner loop iteration, i.e., $\bigO{NImR^2}$ per outer loop iteration.
\end{itemize}

As a result, the overall leading order complexity of TR-TS-ALS is $$\bigO{Nnnz(\tensor{X}) + it \cdot NImR^2},$$
where ``$it$'' denotes the number of outer loop iterations. So, TR-TS-ALS is very applicable to sparse tensor. For dense tensor, similar to TR-KSRFT-ALS and TR-KSRFT-ALS-Premix, the dominant cost of \Cref{alg:tr_ts_als} also  appears in the upfront costs, i.e., the \LineRef{line:ts_TX}.

\begin{remark}
	The leading order computational complexities of the algorithms involved in this paper are summarised in \Cref{table:complexity}.
	
	\begin{table}[htbp] \footnotesize
	\centering
	\caption{ Comparison of leading order computational complexities}
	\label{table:complexity}
	\begin{tabular}{lll}
		\toprule
		Method & Complexity & Ignoring mixing tensor\\
		\midrule
		TR-ALS (\cite{zhao2016TensorRing}) 				& $\bigO{NIR^2 + it \cdot N I^N R^2}$  				& -----\\
		TR-ALS-Sampled (\cite{malik2020SamplingBased})	& $\bigO{NIR^4 + it \cdot NImR^2}$   				& -----\\
		TR-SRFT-ALS (Proposal) & $\bigO{I^N \log(I^N) + it \cdot NmI^2}$ & $\bigO{NIR^2 + it \cdot NI^2 m}$\\
		TR-SRFT-ALS-Premix (Proposal)			& $\bigO{I^N \log(I^N) + it \cdot NImR^2}$ 			& $\bigO{NIR^2 + it \cdot NImR^2}$\\
		TR-TS-ALS (Proposal)			& $\bigO{Nnnz(\tensor{X}) + it \cdot NImR^2}$   				& 
		-----\\
		\bottomrule
	\end{tabular}
\end{table}

	As before,  ``$it$'' in this table denotes the number of outer loop iterations. 
	From \Cref{table:complexity}, we can find that the complexities of TR-SRFT-ALS and TR-SRFT-ALS-Premix are lower than that of the regular TR-ALS, but are higher compared with TR-ALS-Sampled. 
	However, if we don't consider the complexity of mixing the tensor, our methods will have some advantages over TR-ALS-Sampled. 
	For TR-TS-ALS, its complexity can be lower than that of TR-ALS-Sampled for sparse tensor.
\end{remark}

\section{Numerical results}
\label{sec:numerical}
To test our proposed methods, we choose two methods as the baselines.
The first one is TR-ALS, with which we want to show the advantages of the randomized algorithms for TR decomposition for big data. 
The second one is TR-ALS-Sampled from \cite{malik2020SamplingBased} because it is the state-of-the-art randomized method for TR decomposition; see the detailed comparison of this method with others given in \cite{malik2020SamplingBased}. Meanwhile, we have also used the same sampling strategy from TR-ALS-Sampled in TR-KSRFT-ALS and TR-KSRFT-ALS-Premix. 
For TR-ALS and TR-ALS-Sampled, the functions \texttt{tr\_als.m} and \texttt{tr\_als\_sampled.m} from \cite{malik2020SamplingBased} are used in the specific experiments. These two functions are available at \url{https://github.com/OsmanMalik/tr-als-sampled}. Additionally, we also use the MATLAB Tensor Toolbox \cite{tensortoolbox}.

For the initialization of the involved algorithms, as discussed in \cite{battaglino2018PracticalRandomized} on ALS and randomized ALS for CP decomposition, TR-ALS also requires a good starting point to ensure good performance, but the randomized variants gain no obvious advantage from the same initial guess. So, in our experiments, we use random initialization as mentioned in the analysis of computational complexity for randomized algorithms. For TR-ALS, we can choose the results of TR-SVD as the initialization. However, for fairness, we also use random initialization.

In addition, all experiments are run on Matlab R2020b on a computer with an Intel Xeon W-2255 3.7 GHz CPU and 256 GB memory, and all numerical results and plotted quantities are the average over 10 runs.

\subsection{Synthetic data}
\label{ssec:synthetic_data}
Our experimental method for synthetic tensors can be summarized in the following two stages.

\paragraph{Preparation Stage} In this stage, we  use TR-ALS to determine the maximum number of iterations for termination of various algorithms, which will be used as the only termination criterion in \emph{Experimental Stage}.
Specifically, we run TR-ALS until the iterations are larger than $M = 500$ or the relative error is smaller than $\varepsilon = 1 \times 10^{-6}$ and record the iterations $T$. Here, the relative error is defined by 
$$\frac{\| \TR \left( \{\tensor{G}'_n\}_{n=1}^N \right) - \tensor{X} \|_F}{\| \tensor{X} \|_F},$$ 
where the TR-cores $\tensor{G}'_n$ are computed by TR-ALS. This relative error is also used to measure the quality of various randomized algorithms for TR decomposition in the subsequent experiments. This is because we have no way to know the true cores in most applications. 

\paragraph{Experimental Stage} In this stage, we use $2T$ as the terminationcriterion to run TR-ALS/TR-ALS-Sampled/TR-KSRFT-ALS/TR-TS-ALS.
For TR-ALS-Sampled/TR-KSRFT-ALS/TR-TS-ALS, we set the embedding size $m$ to be started at $J_{init}$ and incremented by $J_{inc}$ until $J_{fin}$. 
In the specific experiments, we set $J_{init}=500$, $J_{inc}=250$, and $J_{fin}=5000$.
Thus, we can obtain the corresponding relative errors and running time at each fixed embedding size, and the vectors of relative errors and running time for all the embedding sizes.
As for TR-ALS, it doesn't need the embedding size and its relative error and running time are almost fixed in each experiment. 
Therefore, with respect to the increase of embedding size, we can plot figures to reflect the variation of relative errors and running time of different algorithms.

Four numerical experiments are carried out to test our methods. 
All of our synthetic tensors of size $I \times I \times I$ are generated by creating 3 cores of size $R_{true} \times I \times R_{true}$, i.e., $\tensor{G}_1, \tensor{G}_2$ and $\tensor{G}_3$, and hence the tensor we are trying to recover is $\tensor{X}_{true} = \TR \left( \{\tensor{G}_n\}_{n=1}^3 \right)$. In the specific experiments, we set $R_{true}=10$ and use $R$ to denote the target rank run in the algorithms. To show the robustness and scalability of algorithms, the noise tensor $\tensor{N} \in \bbR^{I \times I \times I}$ with entries drawn from a standard normal distribution is added into the true tensor.
Then the observed tensor is  of the following form 
\begin{equation*}
\tensor{X} = \tensor{X}_{true} + noise\left( \frac{\| \tensor{X}_{true} \|}{\| \tensor{N} \|}\right) \tensor{N},
\end{equation*}
where the parameter $noise$ is the amount of noise. For the sparse tensor, we will only add the noise into the non-zero entries.

In the following, we present the specific way to generate the data and the corresponding numerical results for each of the four experiments.

In the first experiment, we generate a $500 \times 500 \times 500$ tensor as done in \cite{malik2020SamplingBased}. Specifically, 3 cores of size $10 \times 500 \times 10$ with entries drawn independently from a standard normal distribution are first created and then one entry for each core is chosen uniformly and set to 20 to increase the coherence. Finally, the true tensor can be formed using TR decomposition.

 \begin{figure}[htbp] 
 \centering 
 \subfloat[$noise = 0$.]{\includegraphics[scale=0.3]{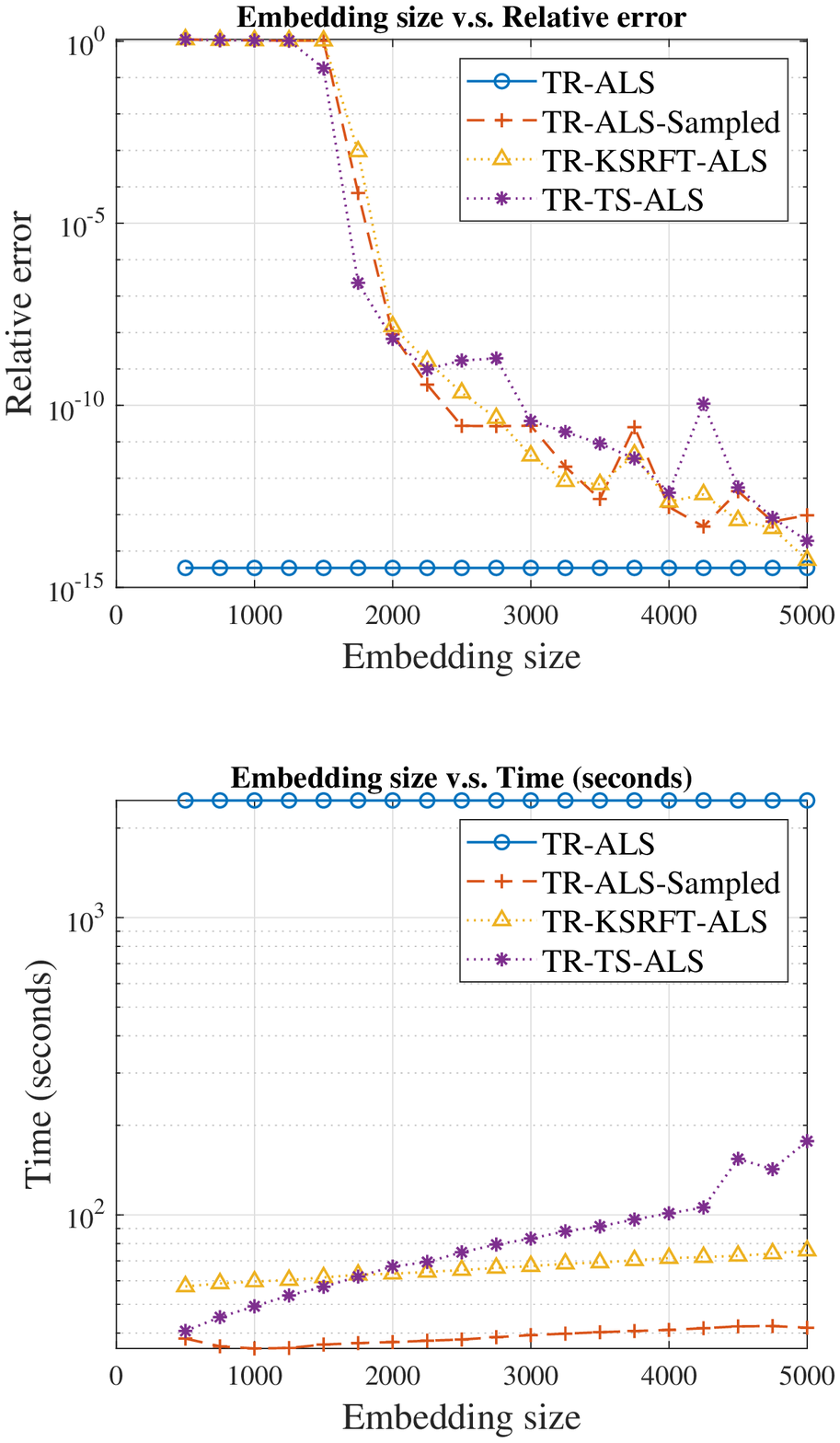}} 
 \subfloat[$noise = 0.01$.]{\includegraphics[scale=0.3]{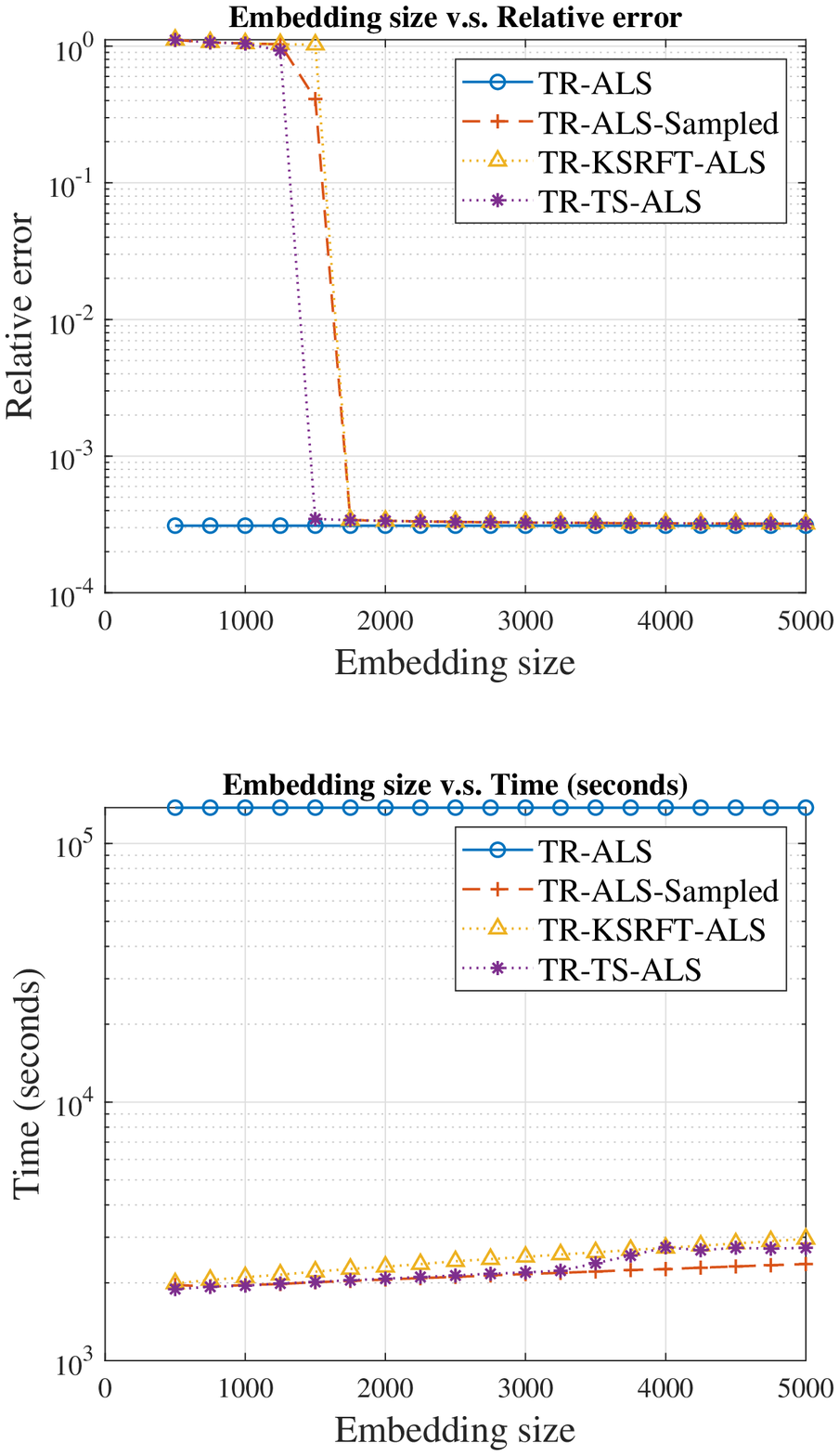}} 
\subfloat[$noise = 0.1$.]{\includegraphics[scale=0.3]{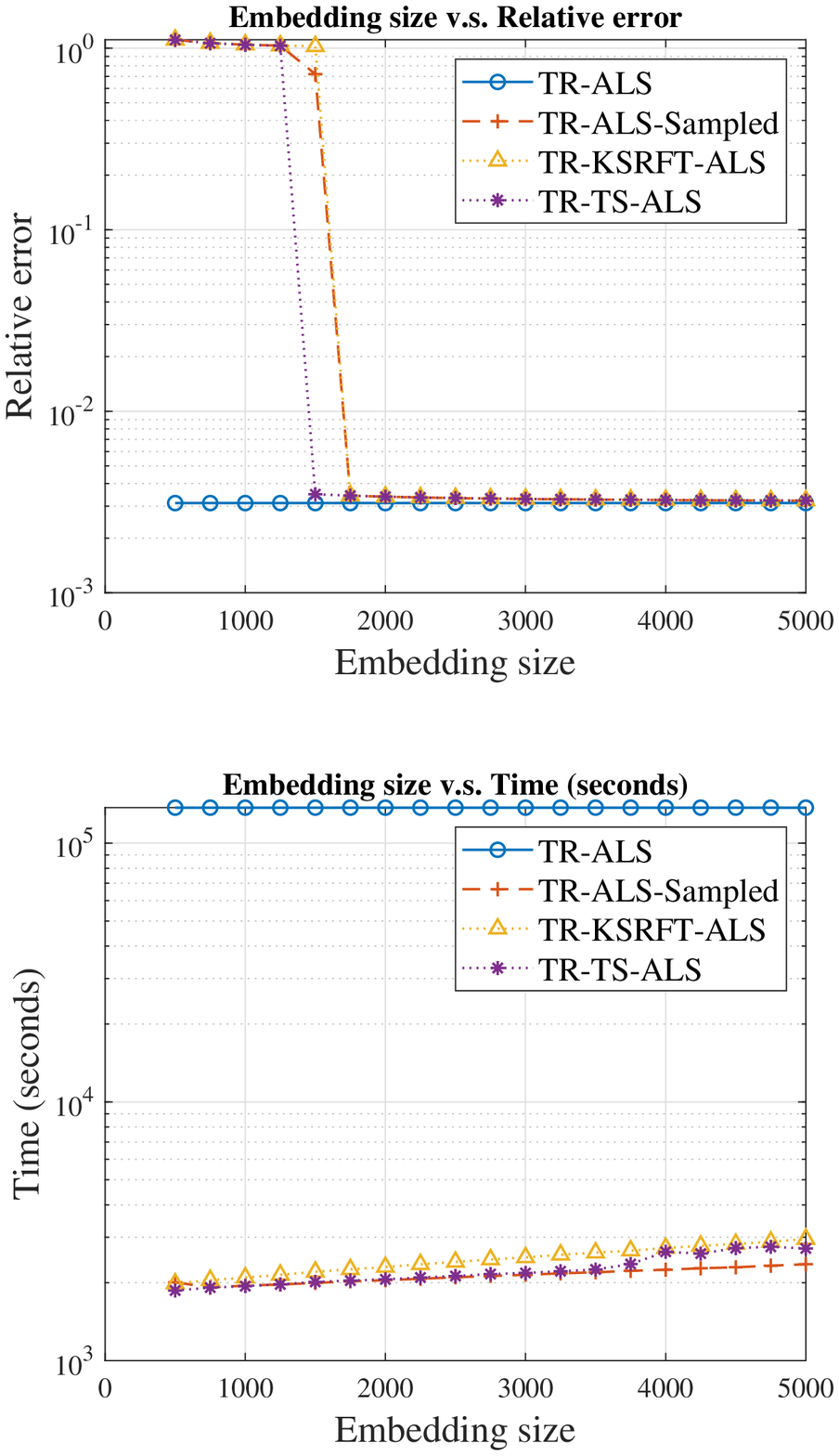}} 
 \caption{Embedding sizes v.s. relative errors and running time (seconds) of the first synthetic experiment with true and target ranks $R_{true}=R=10$ and different noises.} 
 \label{fig:syn1} 
 \end{figure}

\Cref{fig:syn1} shows the numerical results on relative errors and running time for this data with or without noise. 
From the relative errors shown in the top half of the figure, we can see that TR-KSRFT-ALS and TR-TS-ALS match, and in some cases surpass, the performance of TR-ALS-Sampled, and when the embedding size increases to a certain value, these three methods can achieve similar errors as TR-ALS.
For the running time shown in the bottom half of the figure, our two methods and TR-ALS-Sampled have almost the same performance for the case with noise, and are much better than TR-ALS. For the case without noise, our methods perform a little worse than TR-ALS-Sampled, however, still significantly outperform TR-ALS.
In addition, when the noise increases, our algorithms are still efficient, which implies that they may be robust to noise.

In the second experiment, we consider the $500 \times 500 \times 500$ sparse tensor, which is created by 3 cores whose non-zero entries are drawn from a standard normal distribution. Specifically, the 3 cores of size $10 \times 500 \times 10$ are generated by using the Matlab function \texttt{sptenrand([10,500,10], density)}, where $0 \leq \texttt{density} \leq 1$ denotes the sparsity. So the core tensor contains approximately $\texttt{density} \times 50000$ normal distributed non-zero entries. In the specific experiments, we set \texttt{density} = 0.05.

 \begin{figure}[htbp] 
	\centering 
	\subfloat[$noise = 0$.]{\includegraphics[scale=0.3]{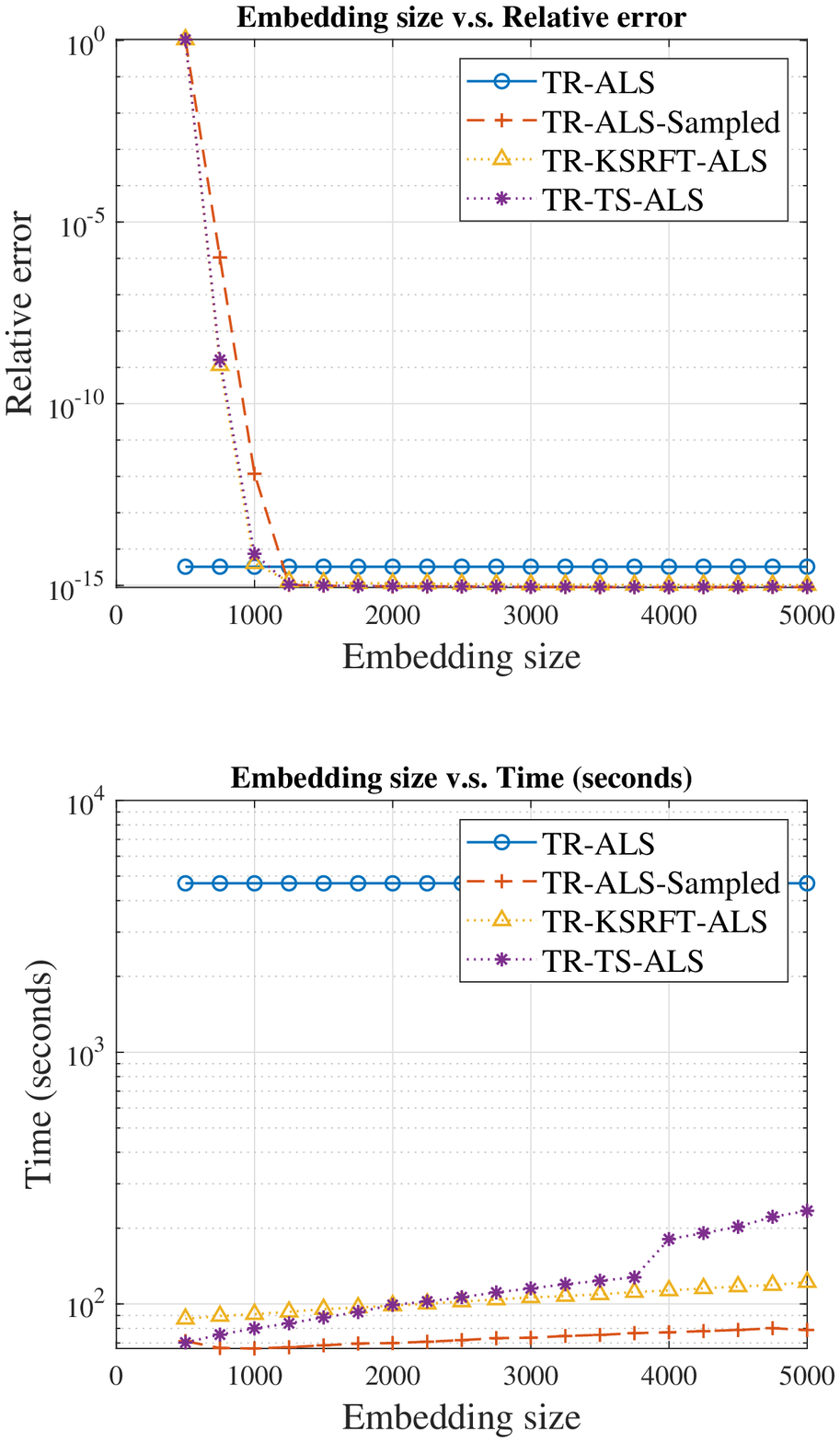}} 
	\subfloat[$noise = 0.01$.]{\includegraphics[scale=0.3]{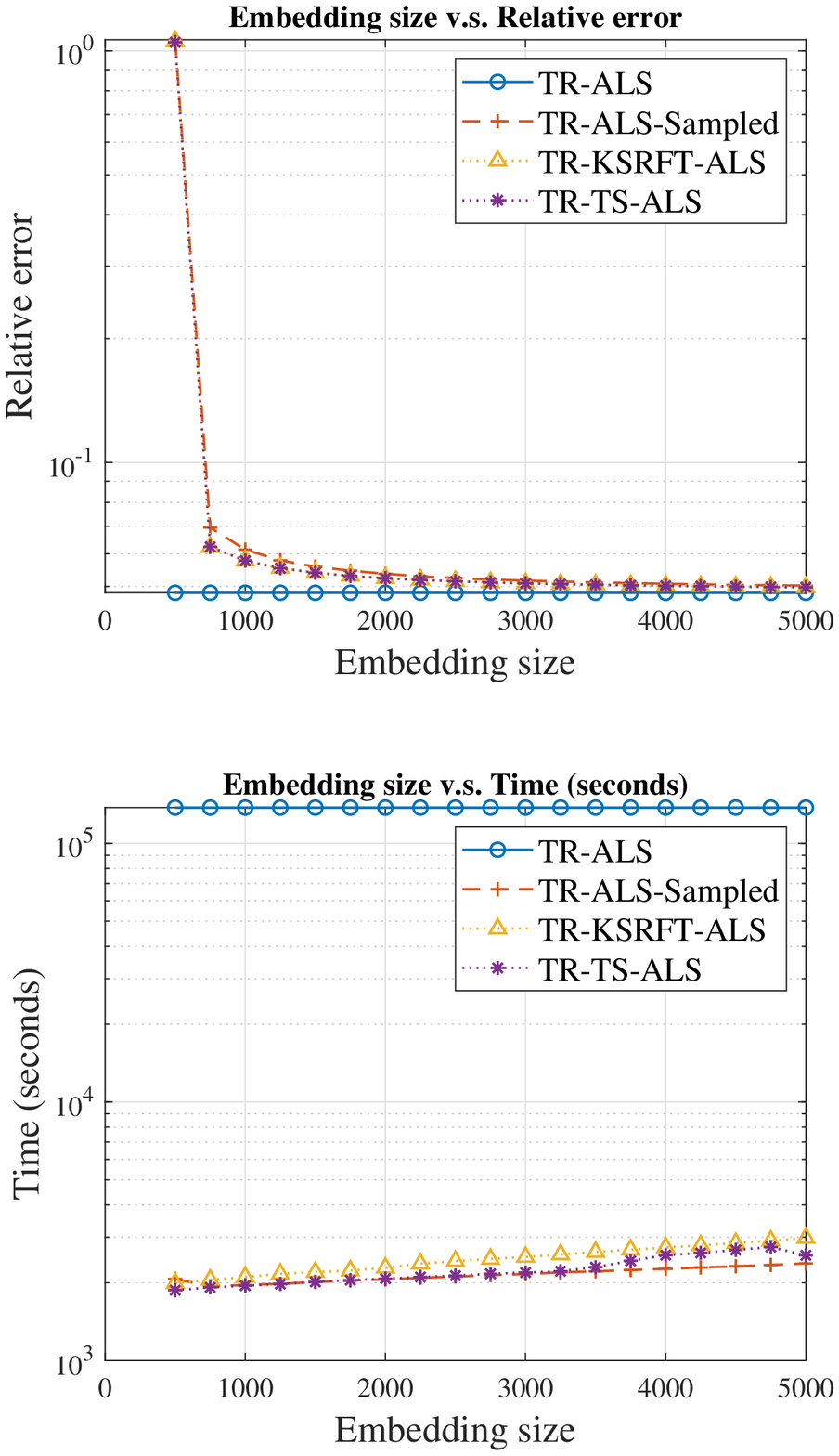}} 
	\subfloat[$noise = 0.1$.]{\includegraphics[scale=0.3]{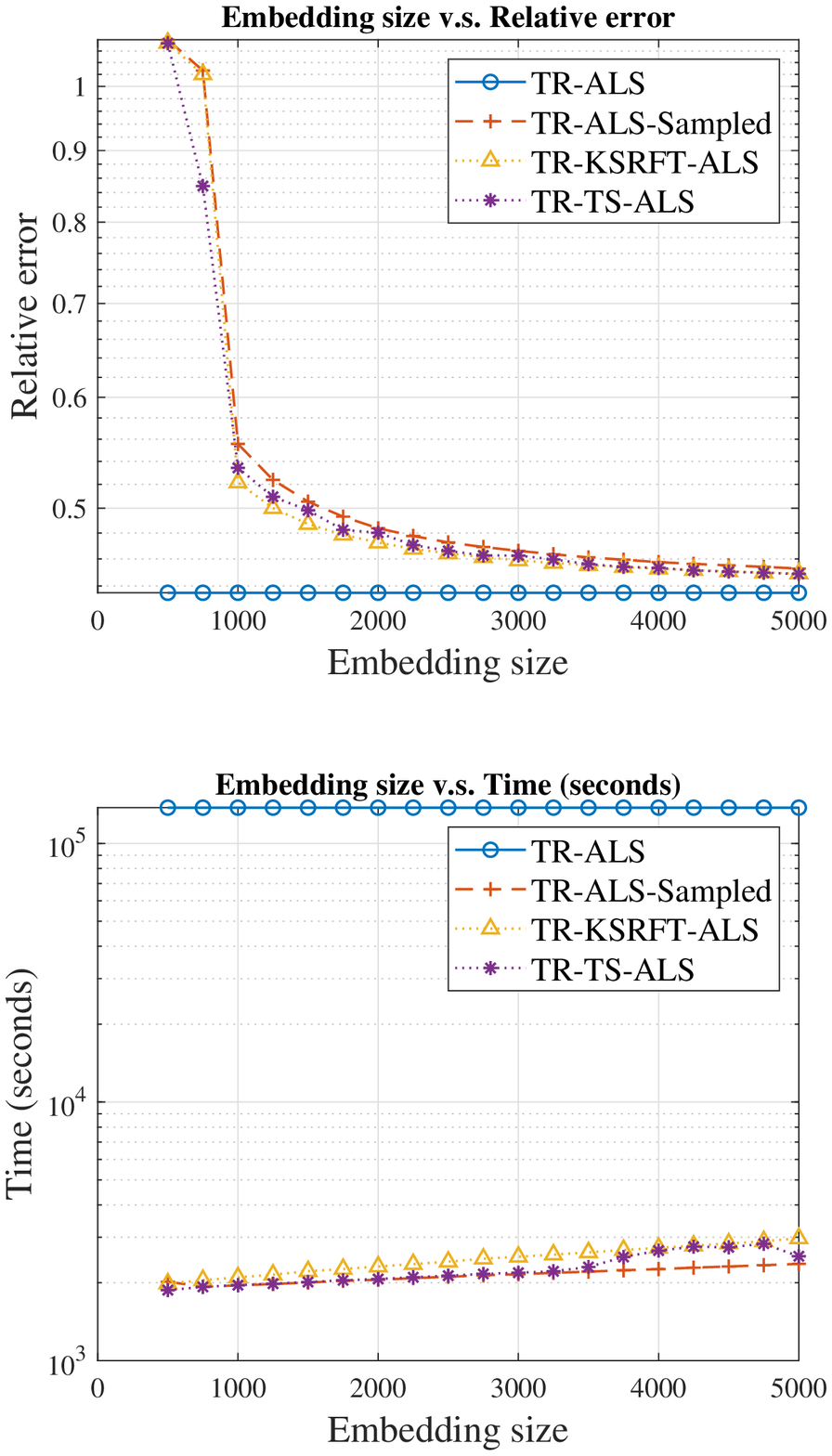}} 
	\caption{Embedding sizes v.s. relative errors and running time (seconds) of the second synthetic experiment with true and target ranks $R_{true}=R=10$ and different noises. }
	\label{fig:syn2} 
\end{figure}

\Cref{fig:syn2} shows the numerical results of this experiment. The conclusions from these results are similar to the ones from the first experiment. That is, TR-KSRFT-ALS, TR-TS-ALS and TR-ALS-Sampled have similar performance in relative errors and running time for the data with noise, and if the data has no noise, our methods are a little worse in running time compared with TR-ALS-Sampled. For all the cases, the above three methods always perform much better than TR-ALS in running time. Furthermore, our methods also show superior robustness. 

In the third experiment, we use the $500 \times 500 \times 500$ tensor generated by a method similar to the one 
in Appendix C.1 of \cite{larsen2020PracticalLeverageBased} which is used to create data for testing the algorithms for CP decomposition. 
Specifically, we first create 3 cores of size $10 \times 500 \times 10$ with entries drawn independently from a standard normal distribution and then
set the first three columns of the mode-2 unfolding matrix of each core to be zero. Finally, a few nonzero elements are added to those zero columns. Hence, there are two user-specified parameters in this kind of data: spread and magnitude. $spread$ means how many nonzero elements are added to each column. For example, $spread$ = 15 means that the first 15 rows in the first column, the second 15 rows in the second column, and the third 15 rows in the third column are nonzero. $magnitude$ means the size of the added nonzero elements. In the specific experiments, we set $spread$ = 15, $magnitude = I/4-10$, and do a change on the third core further. 
That is, we set all the rows except the first 15 ones of the mode-2 unfolding matrix to be zero. 

The numerical results for this data are shown in
\Cref{fig:syn3}. They are a little different from the ones from the previous two experiments.
Our algorithms perform better, especially for the case with noise,  than TR-ALS-Sampled in terms of the relative errors under almost the same running time. Moreover, to achieve the similar errors as TR-ALS, the embedding sizes of our algorithms are much smaller than that of TR-ALS-sampled. 
We plot the leverage scores of the classical mode-2 unfolding matrices of the original and transformed TR-cores in \Cref{fig:box3}, which shows that the original cores indeed have some special fibers. This is consistent with the analysis in \Cref{rem:outlinerr}. 

 \begin{figure}[htbp] 
	\centering 
	\subfloat[$noise = 0$.]{\includegraphics[scale=0.3]{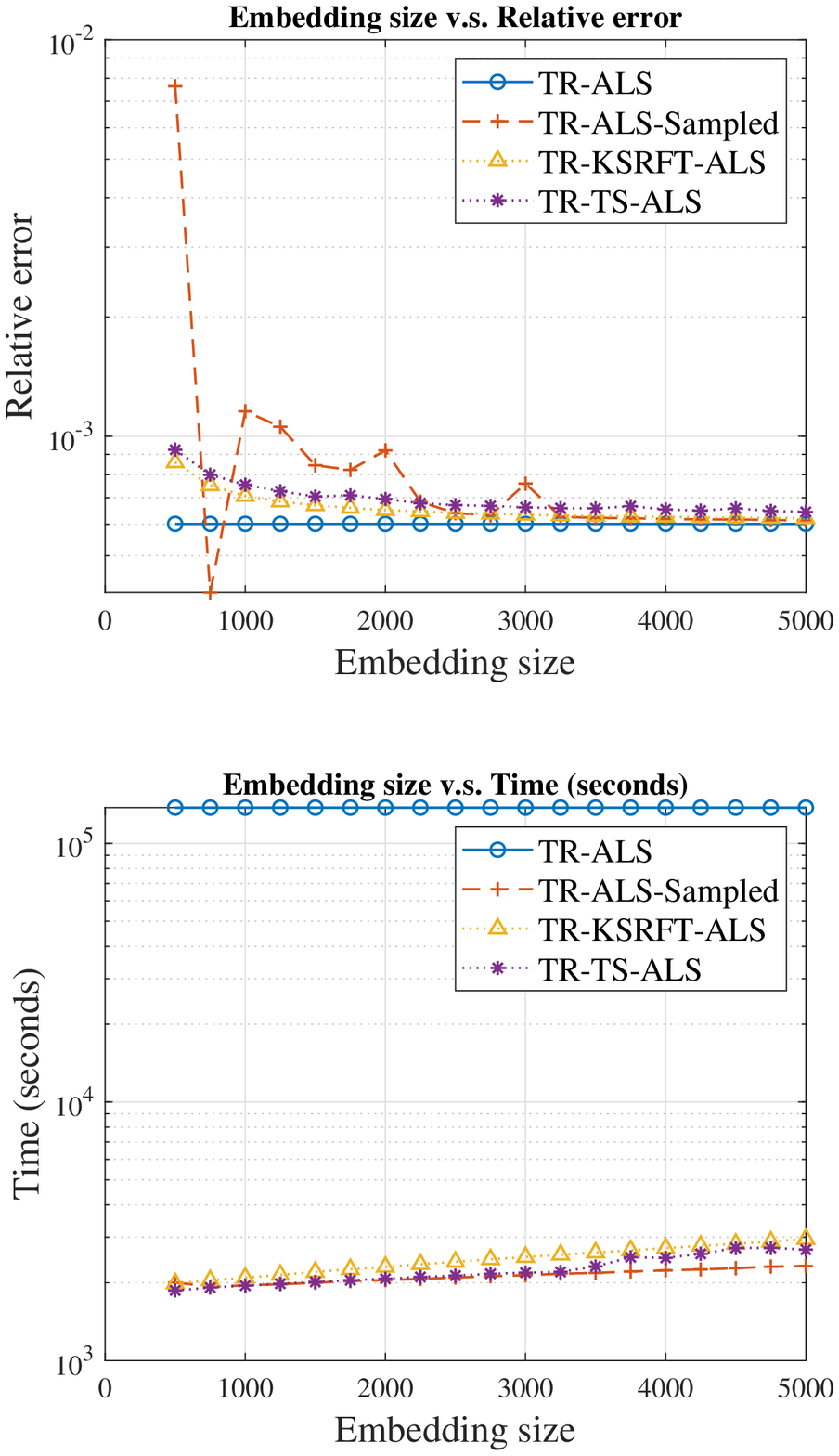}} \label{subfig:syn3a}
	\subfloat[$noise = 0.01$.]{\includegraphics[scale=0.3]{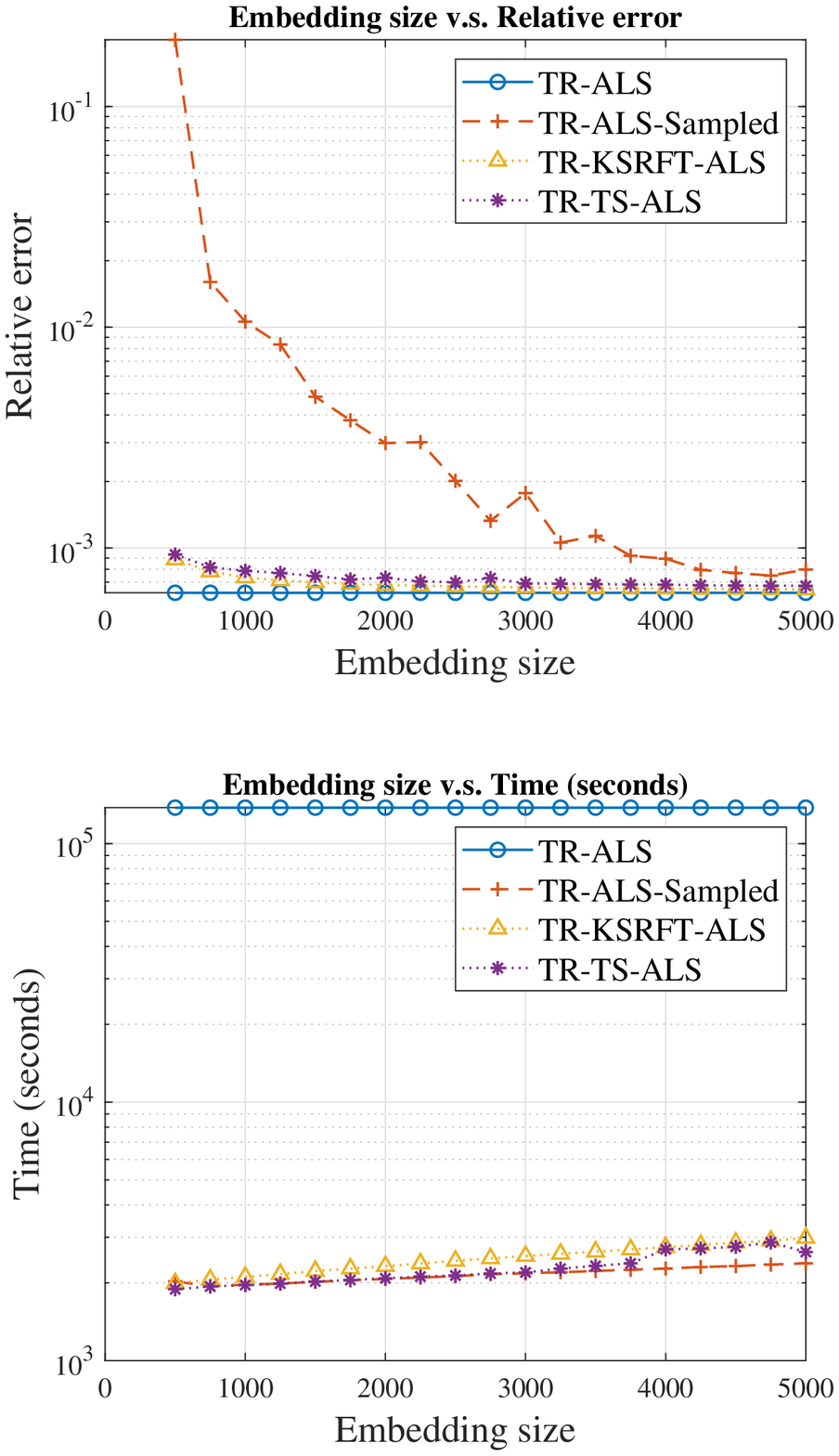}} 
	\subfloat[$noise = 0.1$.]{\includegraphics[scale=0.3]{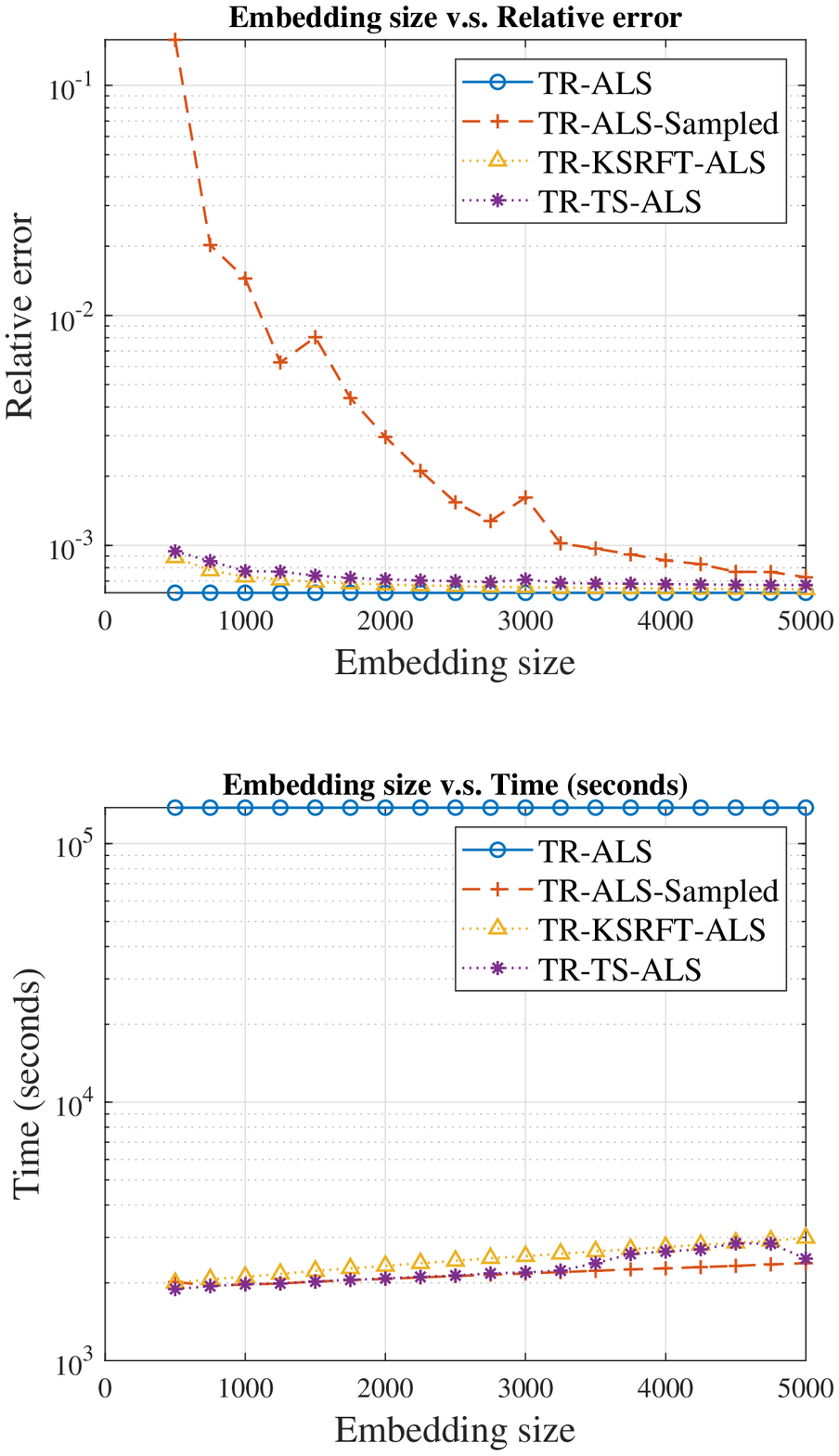}} 
	\caption{Embedding sizes v.s. relative errors and running time (seconds) of the third synthetic experiment with true and target ranks $R_{true}=R=10$ and different noises.}
	\label{fig:syn3} 
\end{figure}

\begin{figure}[htbp]
		\centering
		\includegraphics[width=1 \textwidth]{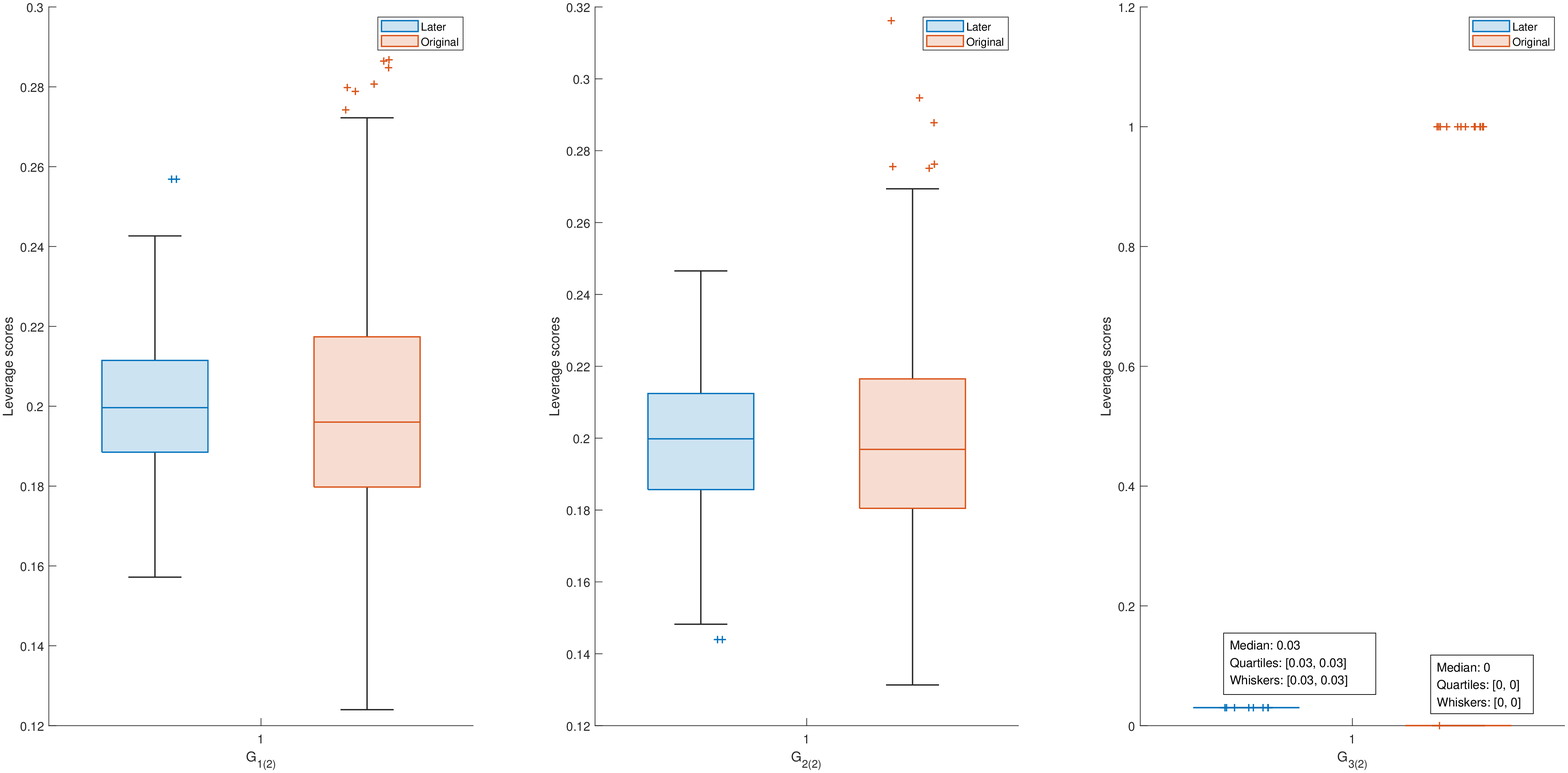}
	\caption{Box plot of leverage scores on the original (right) and transformed (left) TR-cores. }
	\label{fig:box3}
\end{figure}
 
 The fourth experiment is mainly used to test \Cref{alg:pre_tr_srft_als} (TR-KSRFT-ALS-Premix). The complex tensor of size $500 \times 500 \times 500$ is generated in the same way as the first experiment except that the cores are complex. \Cref{fig:syn4} shows the relative errors and running time for this data with different noises, which is in good agreement with the results shown in the first experiment, and the difference between the running time for data without noise is also not remarkable.
  So, we can conclude that \Cref{alg:pre_tr_srft_als} can perform well and is stable under noise for complex data. 
 
 \begin{figure}[htbp] 
	\centering 
	\subfloat[$noise = 0$.]{\includegraphics[scale=0.3]{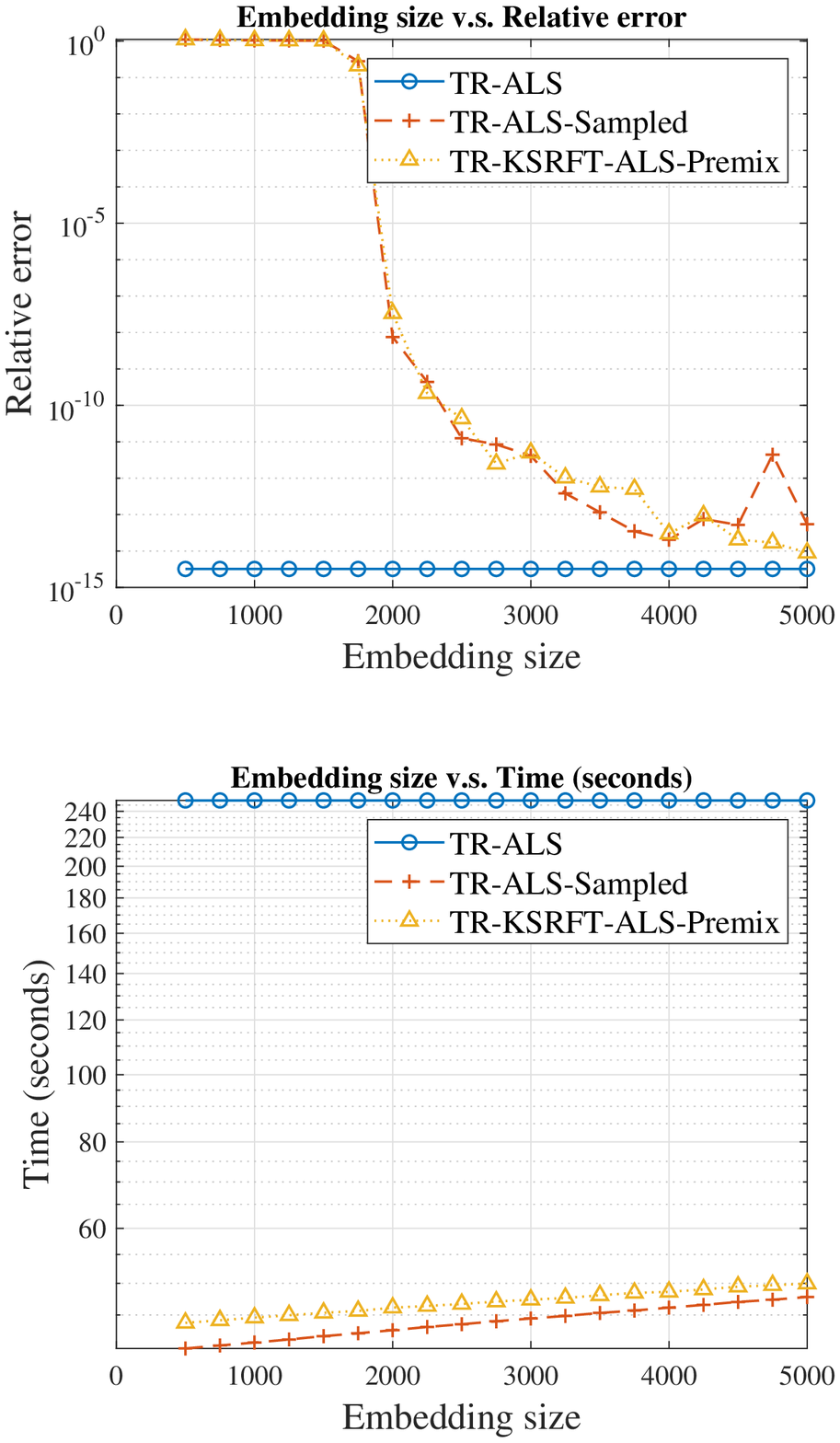}} 
	\subfloat[$noise = 0.01$.]{\includegraphics[scale=0.3]{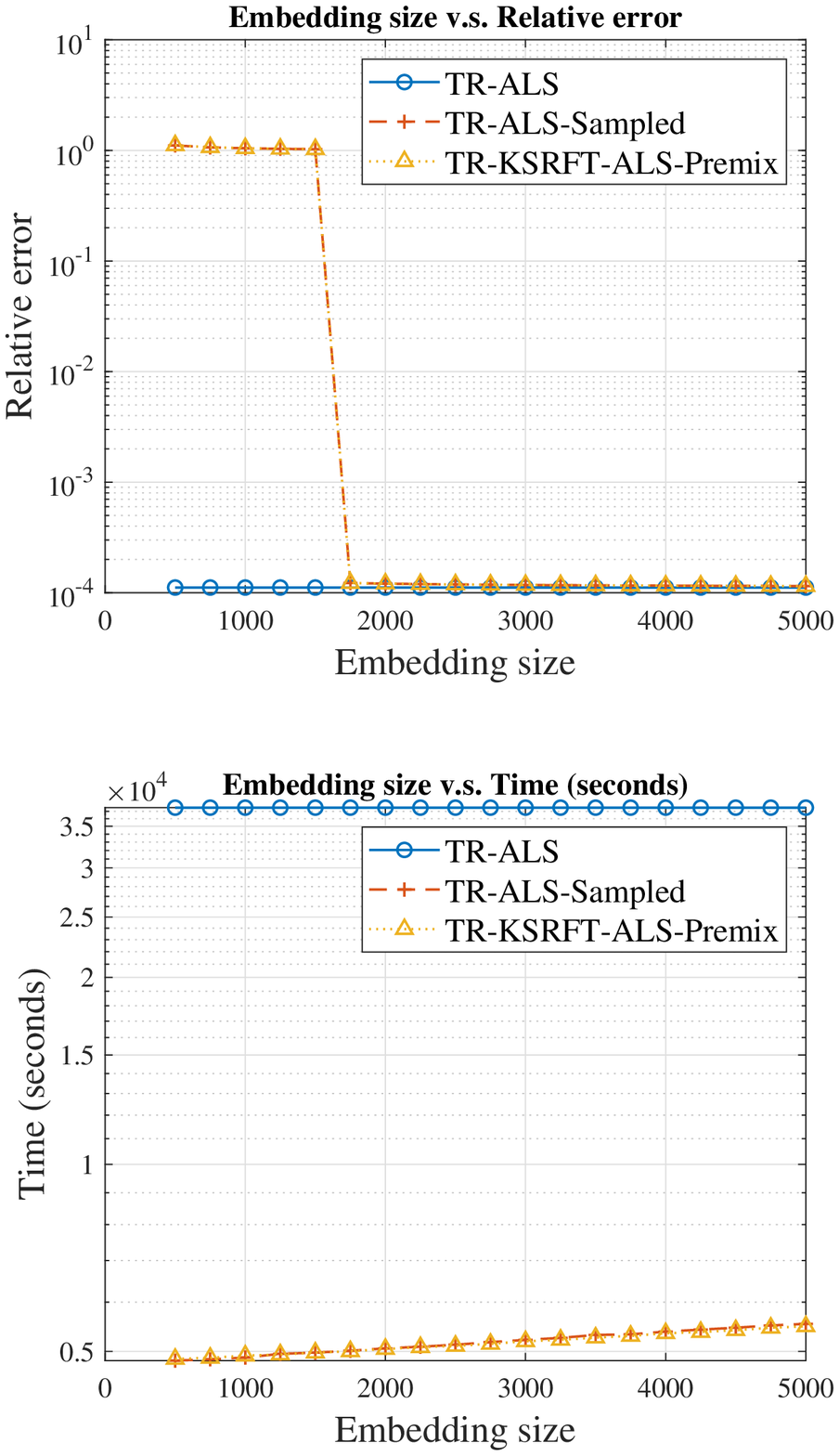}} 
	\subfloat[$R = 10, noise = 0.1$.]{\includegraphics[scale=0.3]{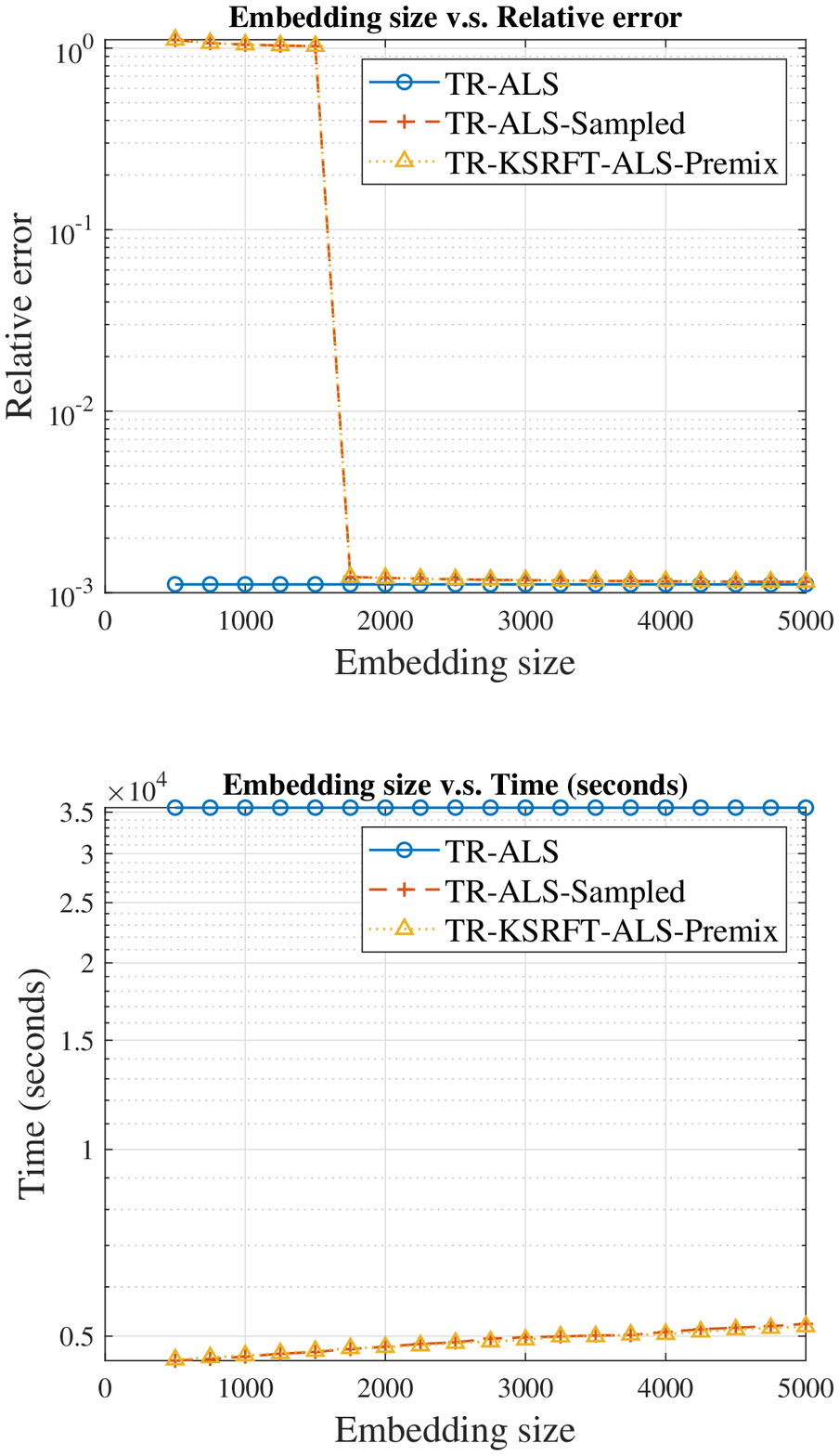}} 
	\caption{Embedding sizes v.s. relative errors and running time (seconds) of the fourth synthetic experiment with true and target ranks $R_{true}=R=10$ and different noises. }
	\label{fig:syn4} 
\end{figure}

\subsection{Real data}
We consider five real data tensors whose brief information is listed in \Cref{table:realdata} including specific links  provided in footnotes.

\begin{table}[htbp]
	\centering
	\caption{Size and type of real datasets.} 
	\label{table:realdata}
	\begin{tabular}{lll}
		\toprule
		Dataset & Size & Type\\
		\midrule
		Indian Pines \tablefootnote{\url{http://www.ehu.eus/ccwintco/index.php?title=Hyperspectral_Remote_Sensing_Scenes} \label{web:indian}} & $145 \times 145 \times 220$ & Hyperspectral \\
		SalinasA. \textsuperscript{\ref{web:indian}} & $83 \times 86 \times 224$ & Hyperspectral \\
		C1-vertebrae \tablefootnote{\url{https://isbweb.org/data/vsj/}} & $512 \times 512 \times 47$ & CT Images \\
		Uber.Hour \tablefootnote{\url{http://frostt.io/} \label{web:uber}} & $183 \times 1140 \times 1717$ & Sparse \\
		Uber.Date \textsuperscript{\ref{web:uber}} & $24 \times 1140 \times 1717$ & Sparse \\
		\bottomrule
	\end{tabular}
\end{table}

The first two datasets are composed of 3-order tensors containing hyperspectral images, where
the first two orders are the image height and width, and the third one is the number of spectral bands. 
The third dataset consists of $512 \times 512 \times 47$ tensors where the $512 \times 512$ frontal slice of each tensor is a 2-dimensional slice of the C1-vertebrae. These images were obtained from the Laboratory of Human Anatomy and Embryology, University of Brussels (ULB), Belgium.
The last two datasets are extracted from the Uber tensor which is a 4-order tensor of size $183 \times 24 \times 1140 \times 1717$ with 3,309,490 nonzeros (0.038\%). The entry $(i, j, k, l)$ denotes the number of pickups on date $i$, hour $j$, at latitude $k$ and longitude $l$, in New York city during the period April–August 2014. 
Due to the limitation of experimental equipment, we divide it into two parts with respect to date and hour respectively, i.e., the Uber.Date and Uber.Hour. The former contains the number of pickups on hour $j$, at latitude $k$ and longitude $l$ for a day, i.e., the entry is $(j, k, l)$, and the latter contains the number of pickups on date $i$, at latitude $k$ and longitude $l$ for an hour, i.e., the entry is $(i, k, l)$.

Considering that the real data always has irregular dimensions, we will not use the same experimental methods as in the second stage in synthetic experiments. Here, we mainly compare the running time and the embedding size $m$ for the same number of iterations and almost the same relative errors. 
Specifically, we first find the maximum iterations according to the method in the first stage of synthetic experiments with $M = 100$ and $\varepsilon = 1 \times 10^{-3}$. Then, for TR-ALS, we can get the running time and relative error; for other algorithms, we try the embedding size $m$ incrementally using the fixed maximum of iterations until the error is close to the 1.1 times of the TR-ALS error. 
The size $m$ is started at  $J_{init} = 2R^2$ and incremented by $J_{inc} = 100$, and hence $J_{fin} = J_{init} + J_{inc} \times num$, where the target rank $R$ is a priori estimate and is different for different data. 
Thus, if we record the number of increases for the embedding size, i.e., $num$, we can report the final embedding size $m$. Meanwhile, we can record the running time and the final error for each algorithm. The specific numerical results are presented in \Cref{tab:real_ex}.

\begin{table}[htbp]
\caption{Decomposition results for real datasets with different target ranks.}
\label{tab:real_ex}
\resizebox{\textwidth}{!}{%
\begin{tabular}{clllll}
\toprule
\multicolumn{2}{c}{Method}                                                                                                                          & TR-ALS    & TR-ALS-Sampled & TR-KSRFT-ALS & TR-TS-ALS \\ \midrule
\multirow{4}{*}{\begin{tabular}[c]{@{}c@{}}Indian Pines \\ ($R = 20$)\end{tabular}} & Error                                                         & 0.0263    & 0.0289         & 0.0289      & 0.0289    \\
                                                                                    & Time                                                          & 32.9536   & 13.7424        & 12.3571     & 12.0229   \\
                                                                                    & $num$                                                         & -         & 120            & 53          & 73        \\
                                                                                    & \begin{tabular}[c]{@{}l@{}}Embedding\\ size ($m$)\end{tabular} & -         & 12800          & 6100        & 8100      \\ \midrule
\multirow{4}{*}{\begin{tabular}[c]{@{}c@{}}SalinasA.\\ ($R = 15$)\end{tabular}}     & Error                                                         & 0.0066    & 0.0069         & 0.0073      & 0.0073    \\
                                                                                    & Time                                                          & 4.0225    & 2.4166         & 1.8510      & 2.2868    \\
                                                                                    & $num$                                                         & -         & 54             & 23          & 30        \\
                                                                                    & \begin{tabular}[c]{@{}l@{}}Embedding\\ size ($m$)\end{tabular} & -         & 5850           & 2750        & 3450      \\ \midrule
\multirow{4}{*}{\begin{tabular}[c]{@{}c@{}}C1-vertebrae \\ ($R = 25$)\end{tabular}} & Error                                                         & 0.0804    & 0.0882         & 0.0883      & 0.0883    \\
                                                                                    & Time                                                          & 409.7951  & 128.3391       & 101.7646    & 156.5089  \\
                                                                                    & $num$                                                         & -         & 228            & 88          & 217       \\
                                                                                    & \begin{tabular}[c]{@{}l@{}}Embedding\\ size ($m$)\end{tabular} & -         & 24050          & 10050       & 22950     \\ \midrule
\multirow{4}{*}{\begin{tabular}[c]{@{}c@{}}Uber.Hour \\ ($R = 15$)\end{tabular}}    & Error                                                         & 0.7530    & 0.8246         & 0.8272      & 0.8274    \\
                                                                                    & Time                                                          & 869.1631  & 64.7240        & 39.0307     & 45.3829   \\
                                                                                    & $num$                                                         & -         & 230            & 40          & 47        \\
                                                                                    & \begin{tabular}[c]{@{}l@{}}Embedding\\ size ($m$)\end{tabular} & -         & 23450          & 4450        & 5150      \\ \midrule
\multirow{4}{*}{\begin{tabular}[c]{@{}c@{}}Uber.Date \\ ($R = 18$)\end{tabular}}    & Error                                                         & 0.3864    & 0.4226         & 0.4246      & 0.4239    \\
                                                                                    & Time                                                          & 1452.1900 & 159.1936       & 51.3584     & 113.8542  \\
                                                                                    & $num$                                                         & -         & 320            & 46          & 147       \\
                                                                                    & \begin{tabular}[c]{@{}l@{}}Embedding\\ size ($m$)\end{tabular} & -         & 32648          & 5248        & 15348     \\ \bottomrule
\end{tabular}%
}
\end{table}

Compared with the results of part of the experiments on synthetic data, the performance of our methods on these real data tensors is a little better. Except for TR-TS-ALS on the data C1-vertebrae which still yields substantial speedups over TR-ALS, our methods always outperform TR-ALS-Sampled in running time under almost the same relative errors. This may be mainly because, compared with TR-ALS-Sampled, the embedding sizes of our algorithms are much smaller.

\section{Concluding remarks}\label{sec:con}
In this paper, we propose two randomized algorithms, i.e., TR-KSRFT-ALS and TR-TS-ALS, for TR decomposition using KSRFT and TensorSketch. They have pretty decent performance in terms of accuracy and computing time compared with the state-of-the-art method, i.e., TR-ALS-Sampled. To achieve these, two new tensor products are defined and their properties are investigated. These new products and properties are of interest in their own right. For example, they can be used to devise the hybrid algorithms for TR decomposition based on stochastic gradient descent and randomized preconditioning. We will present this work in a subsequent paper. Another future work is to try to
reduce the computational cost of the mixing tensor since it usually dominates the total cost of the algorithm.

\bibliographystyle{siamplain}
\bibliography{simods_references}

\begin{thebibliography}{10}

\bibitem{affleck2004ValenceBond}
{\sc I.~Affleck, T.~Kennedy, E.~H. Lieb, and H.~Tasaki}, {\em Valence bond
  ground states in isotropic quantum antiferromagnets}, Comm. Math. Phys., 115
  (1988), pp.~477--528, \url{https://doi.org/10.1007/BF01218021}.

\bibitem{ahmadi-asl2020RandomizedAlgorithms}
{\sc S.~Ahmadi-Asl, A.~Cichocki, A.~H. Phan, M.~G. {Asante-Mensah}, M.~M.
  Ghazani, T.~Tanaka, and I.~V. Oseledets}, {\em Randomized algorithms for fast
  computation of low rank tensor ring model}, Mach. Learn.: Sci. Technol., 2
  (2020), p.~011001, \url{https://doi.org/10.1088/2632-2153/abad87}.

\bibitem{tensortoolbox}
{\sc B.~W. Bader, T.~G. Kolda, et~al.}, {\em Tensor toolbox for matlab}, 2021,
  \url{https://www.tensortoolbox.org} (accessed 2021/04/05).
\newblock Version 3.2.1.

\bibitem{battaglino2018PracticalRandomized}
{\sc C.~Battaglino, G.~Ballard, and T.~G. Kolda}, {\em A practical randomized
  {CP} tensor decomposition}, SIAM J. Matrix Anal. Appl., 39 (2018),
  pp.~876--901, \url{https://doi.org/10.1137/17M1112303}.

\bibitem{che2019RandomizedAlgorithms}
{\sc M.~Che and Y.~Wei}, {\em Randomized algorithms for the approximations of
  {Tucker} and the tensor train decompositions}, Adv. Comput. Math., 45 (2019),
  pp.~395--428, \url{https://doi.org/10.1007/s10444-018-9622-8}.

\bibitem{che2020ComputationLow}
{\sc M.~Che, Y.~Wei, and H.~Yan}, {\em The computation of low multilinear rank
  approximations of tensors via power scheme and random projection}, SIAM J.
  Matrix Anal. Appl., 41 (2020), pp.~605--636,
  \url{https://doi.org/10.1137/19M1237016}.

\bibitem{che2021EfficientRandomized}
{\sc M.~Che, Y.~Wei, and H.~Yan}, {\em An efficient randomized algorithm for
  computing the approximate {Tucker} decomposition}, J. Sci. Comput., 88
  (2021), pp.~1--29, \url{https://doi.org/10.1007/s10915-021-01545-5}.

\bibitem{cichocki2016TensorNetworks}
{\sc A.~Cichocki, N.~Lee, I.~V. Oseledets, A.~H. Phan, Q.~Zhao, and D.~P.
  Mandic}, {\em Tensor networks for dimensionality reduction and large-scale
  optimization: Part 1 low-rank tensor decompositions}, Found. Trends Mach.
  Learn., 9 (2016), pp.~249--429, \url{https://doi.org/10.1561/2200000059}.

\bibitem{cichocki2015TensorDecompositions}
{\sc A.~Cichocki, D.~Mandic, L.~De~Lathauwer, G.~Zhou, Q.~Zhao, C.~Caiafa, and
  H.~A. PHAN}, {\em Tensor decompositions for signal processing applications:
  From two-way to multiway component analysis}, IEEE Signal Process. Mag., 32
  (2015), pp.~145--163, \url{https://doi.org/10.1109/MSP.2013.2297439}.

\bibitem{clarkson2017LowRankApproximation}
{\sc K.~L. Clarkson and D.~P. Woodruff}, {\em Low-rank approximation and
  regression in input sparsity time}, J. ACM, 63 (2017), pp.~54:1--54:45,
  \url{https://doi.org/10.1145/3019134}.

\bibitem{diao2018sketching}
{\sc H.~Diao, Z.~Song, W.~Sun, and D.~P. Woodruff}, {\em Sketching for
  {Kronecker} product regression and p-splines}, in International Conference on
  Artificial Intelligence and Statistics, vol.~84, Playa Blanca, Lanzarote,
  Canary Islands, 2018, PMLR, pp.~1299--1308.

\bibitem{dolgov2014AlternatingMinimal}
{\sc S.~V. Dolgov and D.~V. Savostyanov}, {\em Alternating minimal energy
  methods for linear systems in higher dimensions}, SIAM J. Sci. Comput., 36
  (2014), pp.~A2248 --A2271, \url{https://doi.org/10.1137/140953289}.

\bibitem{drineas2017LecturesRandomized}
{\sc P.~Drineas and M.~W. Mahoney}, {\em Lectures on randomized numerical
  linear algebra}, in The Mathematics of Data, M.~W. Mahoney, J.~C. Duchi, and
  A.~C. Gilbert, eds., vol.~25 of IAS/Park City Mathematics Series,
  AMS/IAS/SIAM, 2018, pp.~1--48, \url{https://doi.org/10.1090/pcms/025/00829}.

\bibitem{drineas2011FasterLeast}
{\sc P.~Drineas, M.~W. Mahoney, S.~Muthukrishnan, and T.~Sarl{\'o}s}, {\em
  Faster least squares approximation}, Numer. Math., 117 (2011), pp.~219--249,
  \url{https://doi.org/10.1007/s00211-010-0331-6}.

\bibitem{gong2021markov}
{\sc T.~Gong, Y.~Dong, H.~Chen, B.~Dong, and C.~Li}, {\em Markov subsampling
  based huber criterion}, 2021, \url{https://arxiv.org/abs/2112.06134}.

\bibitem{gong2020RobustGradientBased}
{\sc T.~Gong, Q.~Xi, and C.~Xu}, {\em Robust gradient-based markov
  subsampling}, in Proceedings of the AAAI Conference on Artificial
  Intelligence, vol.~34(04), New York Hilton Midtown, USA, 2020, AAAI Press,
  pp.~4004--4011.

\bibitem{jin2021FasterJohnson}
{\sc R.~Jin, T.~G. Kolda, and R.~Ward}, {\em Faster {Johnson-Lindenstrauss}
  transforms via {Kronecker} products}, Inf. Inference, 10 (2021),
  pp.~1533--1562, \url{https://doi.org/10.1093/imaiai/iaaa028}.

\bibitem{kolda2009TensorDecompositions}
{\sc T.~G. Kolda and B.~W. Bader}, {\em Tensor decompositions and
  applications}, SIAM Rev., 51 (2009), pp.~455--500,
  \url{https://doi.org/10.1137/07070111X}.

\bibitem{larsen2020PracticalLeverageBased}
{\sc B.~W. Larsen and T.~G. Kolda}, {\em Practical leverage-based sampling for
  low-rank tensor decomposition}, 2020, \url{https://arxiv.org/abs/2006.16438}.

\bibitem{ma2021FastAccurate}
{\sc L.~Ma and E.~Solomonik}, {\em Fast and accurate randomized algorithms for
  low-rank tensor decompositions}, in Advances in Neural Information Processing
  Systems, M.~Ranzato, A.~Beygelzimer, Y.~Dauphin, P.~Liang, and J.~W. Vaughan,
  eds., vol.~34, Curran Associates, Inc., 2021, pp.~24299--24312.

\bibitem{malik2021MoreEfficient}
{\sc O.~A. Malik}, {\em More efficient sampling for tensor decomposition},
  2021, \url{https://arxiv.org/abs/2110.07631}.

\bibitem{malik2018LowRankTucker}
{\sc O.~A. Malik and S.~Becker}, {\em Low-rank {Tucker} decomposition of large
  tensors using {TensorSketch}}, in Advances in Neural Information Processing
  Systems, vol.~31, Montr\'{e}al, Canada, 2018, Curran Associates, Inc.,
  pp.~10117--10127.

\bibitem{malik2020FastRandomized}
{\sc O.~A. Malik and S.~Becker}, {\em Fast randomized matrix and tensor
  interpolative decomposition using {CountSketch}}, Adv. Comput. Math., 46
  (2020), \url{https://doi.org/10.1007/s10444-020-09816-9}.

\bibitem{malik2020GuaranteesKronecker}
{\sc O.~A. Malik and S.~Becker}, {\em Guarantees for the {Kronecker} fast
  {Johnson-Lindenstrauss} transform using a coherence and sampling argument},
  Linear Algebra Appl., 602 (2020), pp.~120--137,
  \url{https://doi.org/10.1016/j.laa.2020.05.004}.

\bibitem{malik2020SamplingBased}
{\sc O.~A. Malik and S.~Becker}, {\em A sampling-based method for tensor ring
  decomposition}, in Proceedings of the 38th International Conference on
  Machine Learning, vol.~139, Virtual Event, 2021, PMLR, pp.~7400--7411.

\bibitem{martinsson2020RandomizedNumerical}
{\sc P.-G. Martinsson and J.~A. Tropp}, {\em Randomized numerical linear
  algebra: Foundations and algorithms}, Acta Numer., 29 (2020), pp.~403--572,
  \url{https://doi.org/10.1017/S0962492920000021}.

\bibitem{minster2020RandomizedAlgorithms}
{\sc R.~Minster, A.~K. Saibaba, and M.~E. Kilmer}, {\em Randomized algorithms
  for low-rank tensor decompositions in the {Tucker} format}, SIAM J. Math.
  Data Sci., 2 (2020), pp.~189--215, \url{https://doi.org/10.1137/19M1261043}.

\bibitem{oseledets2011TensorTrainDecomposition}
{\sc I.~V. Oseledets}, {\em Tensor-train decomposition}, SIAM J. Sci. Comput.,
  33 (2011), pp.~2295--2317, \url{https://doi.org/10.1137/090752286}.

\bibitem{paghrasmus2013CompressedMatrix}
{\sc R.~Pagh}, {\em Compressed matrix multiplication}, ACM Trans. Comput.
  Theory, 5 (2013), pp.~1--17, \url{https://doi.org/10.1145/2493252.2493254}.

\bibitem{perez-garcia2007MatrixProduct}
{\sc D.~Perez-Garcia, F.~Verstraete, M.~M. Wolf, and J.~I. Cirac}, {\em Matrix
  product state representations}, Quantum Inform. Comput., 7 (2007),
  pp.~401--430.

\bibitem{sidiropoulos2017tensor}
{\sc N.~D. Sidiropoulos, L.~De~Lathauwer, X.~Fu, K.~Huang, E.~E. Papalexakis,
  and C.~Faloutsos}, {\em Tensor decomposition for signal processing and
  machine learning}, IEEE Trans. Signal Process., 65 (2017), pp.~3551--3582,
  \url{https://doi.org/10.1109/TSP.2017.2690524}.

\bibitem{vandecappelle2021NumericalAlgorithms}
{\sc M.~Vandecappelle}, {\em Numerical Algorithms for Tensor Decompositions},
  PhD thesis, Arenberg Doctoral School, Faculty of Engineering Science, 2021.

\bibitem{woodruff2014SketchingTool}
{\sc D.~P. Woodruff}, {\em Sketching as a tool for numerical linear algebra},
  Found. Trends Theor. Comput. Sci., 10 (2014), pp.~1--157,
  \url{https://doi.org/10.1561/0400000060}.

\bibitem{yuan2019RandomizedTensor}
{\sc L.~Yuan, C.~Li, J.~Cao, and Q.~Zhao}, {\em Randomized tensor ring
  decomposition and its application to large-scale data reconstruction}, in
  ICASSP 2019 - 2019 IEEE International Conference on Acoustics, Speech and
  Signal Processing (ICASSP), Brighton Conference Centre Brighton, U.K., 2019,
  IEEE, pp.~2127--2131.

\bibitem{zhao2016TensorRing}
{\sc Q.~Zhao, G.~Zhou, S.~Xie, L.~Zhang, and A.~Cichocki}, {\em Tensor ring
  decomposition}, 2016, \url{https://arxiv.org/abs/1606.05535}.

\bibitem{zhou2014DecompositionBig}
{\sc G.~Zhou, A.~Cichocki, and S.~Xie}, {\em Decomposition of big tensors with
  low multilinear rank}, 2014, \url{https://arxiv.org/abs/1412.1885}.

\end{thebibliography}
\end{sloppypar}
\end{document}